\documentclass[letterpaper,10pt,conference]{IEEEtran}

\IEEEoverridecommandlockouts
\usepackage{cite}
\usepackage{amsmath,amsthm,amssymb,array,cases,mathrsfs,subcaption,here,enumitem,ascmac,url,cite,float,comment,color}
\usepackage{algorithmic}
\usepackage{graphicx}
\usepackage{textcomp}
\usepackage{xcolor}

\theoremstyle{plain}
\newtheorem{theorem}{Theorem}
\newtheorem{lemma}{Lemma}
\newtheorem{proposition}{Proposition}
\newtheorem{corollary}{Corollary}
\newtheorem{remark}{Remark}
\newtheorem{assumption}{Assumption}

\theoremstyle{definition}
\newtheorem{definition}{Definition}

\usepackage{hyperref}

\begin{document}

\title{\LARGE
Koopman Operators for Global Analysis of Hybrid Limit-Cycling Systems: Construction and Spectral Properties\\
\thanks{The work was partially supported by JSPS KAKENHI (Grant No. 23H01434) and JSPS Bilateral Collaborations (Grant No. JPJSBP120242202). 
}
}

\author{Natsuki~Katayama$^1$\thanks{$^1$Natsuki Katayama and Yoshihiko Susuki are with the Department of Electrical Engineering, Kyoto University, Katsura, Nishikyo-ku, Kyoto 615-8510, Japan. E-mails: \texttt{n-katayama@dove.kuee.kyoto-u.ac.jp}, \texttt{susuki.yoshihiko.5c@kyoto-u.ac.jp}} and Yoshihiko~Susuki$^1$}

\maketitle

\begin{abstract}
This paper reports a theory of Koopman operators for a class of hybrid dynamical systems with 
globally asymptotically stable periodic orbits, called 
hybrid limit-cycling systems. 
We leverage 
smooth structures intrinsic to the hybrid dynamical systems, thereby extending the existing theory of Koopman operators 
for smooth dynamical systems. 
Rigorous construction of an observable space is carried out to preserve the 
inherited smooth structures of the hybrid dynamical systems. 
Complete spectral characterization of the Koopman operators acting on the constructed space is then derived where the 
existence and uniqueness of their eigenfunctions 
are ensured. 
Our results facilitate 
global analysis of hybrid dynamical systems using the Koopman operator.
\end{abstract}


\IEEEpeerreviewmaketitle

\section{Introduction}
Hybrid dynamical systems 
are the form of mathematical modeling of complex dynamics involving both continuous flows and discrete transitions \cite{van2007introduction,goebel2009hybrid}. 
Global theory of hybrid dynamical systems, as a generalization of that of smooth dynamical systems, is still challenging in the modern systems and control theory: see, e.g., \cite{simic2001structural,simic2002hybrid}. 
In this paper, we develop a global theory of hybrid dynamical systems with 
asymptotically stable periodic orbits in the Koopman operator framework \cite{mauroy2020koopman}. 

The Koopman operator for a dynamical system, defined as a linear operator on observables (functions on the 
state space of the system), offers a linear perspective that captures global information (all evolutions of states) of the system through its spectral properties. 
Its eigenfunctions reveal geometric properties of the state space
globally \cite{mauroy2013isostable}: for a system 
with an asymptotically stable periodic orbit, one 
specific eigenfunction captures the so-called 
isochrons \cite{mauroy2012use}, while others capture novel 
geometric features \cite{shirasaka2017phase_isostable,monga2019phase} including isostables \cite{wilson2016isostable}. 
These results are established only for smooth dynamical systems, where smooth eigenfunctions are guaranteed to exist uniquely \cite{kvalheim2021existence,mezic2020spectrum}. 
No theoretical research of Koopman operators for global analysis of non-smooth or hybrid dynamical systems is reported although their statistical properties, such as erogodicity, have been studied using the Koopman operator in \cite{govindarajan2016operator,gerlach2020koopman}.

Our research aims to develop a  theory of Koopman operators for hybrid dynamical systems, which is parallel to that for 
smooth dynamical systems (see, e.g., \cite{kvalheim2021existence}). 
In this paper, to leverage the smooth Koopman operator theory for systems with globally asymptotically stable limit cycles \cite{kvalheim2021existence}, called limit-cycling systems, we focus on 
hybrid dynamical systems with globally asymptotically stable periodic orbits \cite{simic2002hybrid,shirasaka2017phase_hybrid}, called \emph{hybrid} limit-cycling systems. 
The hybrid limit-cycling systems are prevalent in many physical and engineering applications, including robotic locomotion \cite{grizzle2001asymptotically,morris2005restricted}, power generation \cite{hiskens2007switching}, and neural spiking \cite{lou2015results}.
Technically, we focus on the approach 
developed in \cite{simic2005towards,burden2015model}, which introduce as a generalization of smooth manifolds of smooth dynamical systems the so-called \emph{hybrifold}. 
This approach involves to \emph{glue together manifolds} corresponding to state spaces of individual modes of a hybrid dynamical system, each with locally smooth dynamics, thereby enabling us to extend the theory of smooth systems to hybrid ones.
%

This paper reports the theoretical foundation of Koopman operators for hybrid limit-cycling systems.
We construct a (semi-)group of Koopman operators with suitable spectral properties while preserving a smooth structure of the hybrifold. 
Our first contribution is to present a rigorous treatment of the choice of observable space on which the Koopman operators act, and that preserves the smooth structure of the hybridfold.
Within this chosen space, we present their complete spectral characterization that is our second contribution. 
Specifically,
the existence and uniqueness of eigenfunctions of the Koopman operators are presented. 
The spectral characterization leads the existence result of linear embeddings for the hybrid limit-cycling systems. 

The rest of this paper is organized as follows: 
In Section~\ref{sec:Hybrid_intro}, we introduce a hybrid dynamical system with a hybrid limit cycle. 
In Section~\ref{sec:Koopman}, we define the Koopman operators and describe the structure of their observable space.
In Section~\ref{sec:KEF}, we demonstrate the existence and uniqueness of the eigenfunctions of the Koopman operators. 
An illustrative example is introduced in Section~\ref{sec:example}. 
Conclusions are presented in Section~\ref{sec:conclusion} with a brief summary and future directions. 
The proofs of lemmas and theorems are included in Appendix \ref{app:proofs}.

\section{Hybrid Limit-Cycling Systems}
\label{sec:Hybrid_intro}

In this section, we introduce a hybrid dynamical system based on the work on \cite{burden2015model}. 
To study the qualitative behaviors near periodic orbits of hybrid systems, the authors of \cite{burden2015model} define a domain as a disjoint set of manifolds, where each manifold corresponds to the mode. 
Because the discrete transition is trivial under the assumption that the modes change periodically, they omit detailed notation concerning the discrete variables. 
Following their notations in \cite{burden2015model}, we introduce the deterministic \emph{hybrid dynamical system} $H=(J,M,F,G,R)$ as follows: 
\begin{itemize}
    \item $J=\{1,2,\ldots,|J|\}$ is a set of discrete-valued variables. $j\in J$ is an equivalent class operated by modulo $|J|$; 
    \item $M=\amalg_{j\in J} M^{(j)}$ is a set of continuous-valued variables, where $M^{(j)}$ is an $n$-dimensional ${C}^r$ manifold with boundary $\partial M^{(j)}$; 
    \item $F:M\to {\rm T}M$ is a ${C}^r$ vector field, where ${\rm T}M$ is a tangent bundle; 
    \item $G=\amalg_{j\in J} G^{(j)}$ is a guard, where $G^{(j)} \subset \partial M^{(j)}$ is a codimension-1 ${C}^r$ manifold; 
    \item $R:G\to \partial M$ ($\partial M:= \amalg_{j\in J} \partial M^{(j)}$) is a ${C}^r$-class reset map. 
\end{itemize}

Here, we investigate $H$ using the gluing technique \cite{simic2005towards}, which requires appropriate assumptions. 
Furthermore, because we will define a flow similarly to smooth dynamical systems, it is necessary to guarantee executions from arbitrary initial states over $[0,+\infty)$. 
Therefore, as in \cite{burden2015model,simic2005towards}, we make the following assumptions:
\begin{assumption}
    \label{ass:transverseG}
    $F$ is strictly outward-pointing \cite[p118]{Lee_smooth} on $G$. 
\end{assumption}
\begin{assumption}
    \label{ass:transverseRG}
    For all $j\in J$, $R|_{G^{(j)}}:G^{(j)}\to \partial M^{(j+1)}$ is diffeomorphic onto its image and $F^{(j+1)}$ is strictly inward-pointing \cite[p118]{Lee_smooth} on $R(G^{(j)})$. 
\end{assumption}
\begin{assumption}
    \label{ass:Nointersect}
    For all $j\in J$, there is no intersection between the closure of $G^{(j)}$ and the closure of $R(G^{(j-1)})$. 
\end{assumption}
\begin{assumption}
    \label{ass:finitejump}
    All of the maximal integral curve \cite[p212]{Lee_smooth} from $x\in M^{(j)}$ reach $G^{(j)}$ for all $j\in J$. 
\end{assumption}
Here, we explain the meaning of these assumptions. 
Assumptions \ref{ass:transverseG} and \ref{ass:transverseRG} are necessary to glue the manifolds via the reset map $R$. 
Additionally, Assumption \ref{ass:transverseG} ensures the uniqueness of executions. 
Assumptions \ref{ass:Nointersect} and \ref{ass:finitejump} are required to guarantee the execution over $[0,\infty)$. 
More specifically, Assumption \ref{ass:Nointersect} excludes the so-called Zeno phenomenon \cite{morris2005restricted} and Assumption \ref{ass:finitejump} excludes the cases in which an integral curve reaches $\partial M^{(j)} \backslash G^{(j)}$, where no reset map is defined. 
Under these assumptions, one can define the execution $x(t)$ of the hybrid dynamical system $H$ over $t\in [0,+\infty)$. 
\begin{definition}
    \label{def:execution}
    An \emph{execution} of the hybrid dynamical system $H$ is a right-continuous function $x:[0,+\infty)\to M$ such that; 
    \begin{enumerate}
        \item if $x$ is continuous at $t$, then $x$ is differentiable at $t$ and $\left. \frac{{\rm d}}{{\rm d}s}x(s) \right|_{s = t} = F(x(t))$; 
        \item if $x$ is discontinuous at $t$, then the limit $x^-(t)=\lim_{s\uparrow t} x(s)$ exists, $x^-(t)\in G$ and $x(t)=R(x^-(t))$.
    \end{enumerate}
\end{definition}
For each $j\in J$, since the smooth manifold $M^{(j)}$ and the smooth vector field $F^{(j)} = F|_{M^{(j)}}$ are introduced, the associated smooth local flow $\varphi_t^{(j)}$ can be defined. 
$\varphi_t^{(j)}$ is $C^r$ with respect to $t$ and $x$ because $F^{(j)}$ is $C^r$.
Here, we define a neighborhood of $G^{(j)}$ as $N^{(j)}\in M^{(j)}$ where all integral curves from $x \in N^{(j)}$ reach to $G^{(j)}$. 
$N^{(j)}$, called a \emph{collar neighborhood} \cite[p222]{Lee_smooth}, is isomorphic to $G^{(j)}\times [0,1)$ and is ensured to exist from Assumptions \ref{ass:Nointersect} and \ref{ass:finitejump}.
We denote the time it takes for the solution from the initial state $x\in N^{(j)}$ to reach $G^{(j)}$ as $\sigma^{(j)}(x)$
Then, $\sigma^{(j)}:N^{(j)}\to \mathbb{R}$, called the time-to-impact map \cite[Lemma 2]{burden2015model}, is $C^r$ map from Assumption \ref{ass:transverseG} and the inverse function theorem. 
Additionally, let $h^{(j)}(x) := \varphi_{\sigma^{(j)} (x)}^{(j)}(x)$ which projects from $N^{(j)}$ to $G^{(j)}$ along trajectories. 
This is clearly $C^r$. 
We summarize the introduced notations as in Fig \ref{fig:time-to-impact}. 
\begin{figure}
    \centering
    \includegraphics[width=0.8\linewidth]{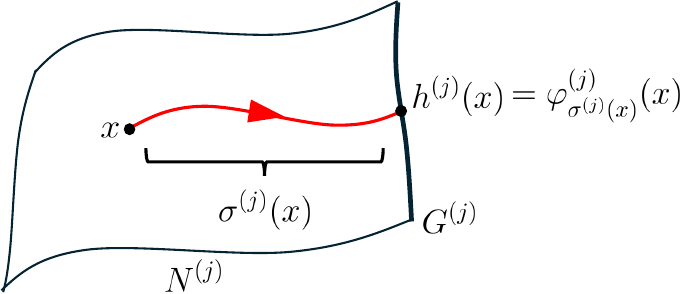}
    \caption{time-to-impact map $\sigma^{(j)}$ and projection $h^{(j)}$. }
    \label{fig:time-to-impact}
\end{figure}

Since the execution $x(t)$ is defined over $[0,+\infty)$ for all initial states, one can define the global flow. 
\begin{definition}
    \label{def:flow}
   Let $\varphi_t (x)$ an execution from an initial state $x\in M$. 
   Then, $\varphi_t:M\to M$ defines the \emph{flow} of the hybrid dynamical system $H$.
\end{definition}
It follows that $\{ \varphi_t \}_{t>0}$ forms a 1-parameter semi-group, implying that $\varphi_{t+s} = \varphi_t \circ \varphi_s$ for $t,s>0$ and $\varphi_0 = {\rm i.d.}$.
Under the above definitions, the following lemma holds and plays a key role to derive our results (see also \cite[Theorem 3]{burden2015model} on which our proof is based). 
\begin{lemma}
\label{lem:smoothing}
    Let $H$ be a hybrid dynamical system satisfying Assumptions \ref{ass:transverseG}-\ref{ass:finitejump} and let $\varphi_t$ be its flow.
    Then, there exists a triple $n$-dimensional $C^r$ manifold $\tilde{M}$, $C^r$ flow $\{ \tilde{\varphi}_t :\tilde{M} \to \tilde{M}\}_{t>0}$, and $C^r$ surjective map $\pi:M\to \tilde{M}$ such that; 
    \begin{equation}
        \label{eq:conjugacy}
        \textstyle
         \tilde{\varphi}_t \circ \pi (x) = \pi \circ \varphi_t (x),\quad
        \forall x\in M, ~ \forall t \in [0,\infty).
    \end{equation}
    Furthermore, for all $j\in J$, the restriction ${\pi}|_{M^{(j)}}$ is a ${C}^r$ diffeomorphism. 
\end{lemma}
\begin{remark}
In \cite{simic2005towards}, $\tilde{M}$ is called a \emph{hybrifold} and $\tilde{\varphi}_t$ a \emph{hybrid flow}. 
\end{remark}

Here, we introduce the notion of hybrid limit cycle. 
First, if there exist $x\in M$ and $\tau > 0$ such that $\varphi_\tau (x) = x$, then the set $\{ \varphi_t (x) \}_{t\in [0,\tau)}$ is called a periodic orbit. 
Especially, if there is no smaller positive number than $\tau$, $\tau$ is called a period. 
The periodic orbit is called a \emph{globally asymptotically stable hybrid limit cycle}
if it attracts all executions that start from the outside of it (see Appendix \ref{app:definition} for the precise definition). 
In this paper, we investigate hybrid dynamical systems with globally asymptotically stable hybrid limit cycles. 
Thus, the following assumption is added: 
\begin{assumption}
    \label{ass:GESHLC}
    $H$ has a globally asymptotically stable hybrid limit cycle.
\end{assumption}

\section{
Construction of Koopman Operators
}
\label{sec:Koopman}
In this section, we construct the Koopman operators that inherits the piecewise smooth structures of the hybrid system $H$. 
First of all, we introduce the Koopman operators for the flow $\varphi_t$ of the hybrid system $H$. 
\begin{definition}
\label{def:Koopman}
Consider a space $\mathcal{F}$ of observable $f:M\to \mathbb{C}$. 
The family of Koopman operators $U_t:\mathcal{F}\to \mathcal{F}$ associated with the flow $\varphi_t$ $(t\geq 0)$ is defined through the composition:
\begin{equation}
\label{eq:Koopmandef}
    U_t f (x) := f\circ \varphi_t (x) = f(\varphi_t (x)),\quad \forall f \in \mathcal{F}.
\end{equation}
\end{definition}
One can verify that $U_t$ is a linear operator and the family of Koopman operators, $\{U_t\}_{t>0}$, forms a 1-parameter semi-group. 
Due to its linearity, we define eigenvalues and associated eigenfunctions of the Koopman operators, called Koopman eigenvalues and Koopman eigenfunctions, respectively, which consists of their point spectrum. 
\begin{definition}
    \label{def:KEF}
    If there exists $\lambda\in\mathbb{C}$ and $\phi_\lambda \in \mathcal{F}\backslash \{0\}$ such that
    \begin{equation}
    \label{eq:KEFdef}
    U_t \phi_\lambda = {\rm e}^{\lambda t} \phi_\lambda,\quad \forall t>0,
    \end{equation}
    then $\lambda$ is called a Koopman eigenvalue and $\phi_\lambda$ a Koopman eigenfunction.
\end{definition}

The choice of $\mathcal{F}$ is crucial in ensuring nice spectral properties and in capturing the system's behavior. 
For smooth dynamical systems, the suitable choice of $\mathcal{F}$ is well studied, which guarantees the existence and uniqueness of the Koopman eigenfunctions \cite{kvalheim2021existence,mezic2020spectrum}. 
More specifically, if a $C^r$ dynamical system $(M,\varphi_t)$ has an asymptotically stable equilibrium point or an asymptotically stable limit cycle, then by choosing $\mathcal{F}$ as $C^r$, there exist the Koopman eigenfunctions uniquely, which capture geometric properties of the original dynamical system (see \cite{kvalheim2021existence} for more details). 
Our aim is to choose $\mathcal{F}$ for the hybrid system $H$. 
Since $H$ inherits a $C^r$ structure as claimed by Lemma \ref{lem:smoothing}, it seems possible to select $\mathcal{F}$ such that $\mathcal{F}$ preserves this smooth structure, levaraging the results of \cite{kvalheim2021existence}. 
In particular, if $\mathcal{F}$ is defined such that for all $f\in \mathcal{F}$, $f\circ \pi^{-1} \in C^r (\tilde{M})$, then the Koopman eigenfunctions exist uniquely, based on the fact that two dynamical systems $(M,\varphi_t)$ and $(\pi(M),\pi\circ\varphi_t)$ possess the same eigenvalues \cite{mezic2020spectrum}. 

With the above idea in mind, we now construct the observable space $\mathcal{F}$ that is endowed with a smooth structure. 
First, we define a $C^r$ map $\Psi^{(j)}:N^{(j)} \mapsto \Psi^{(j)} (N^{(j)}) \subset M^{(j+1)}$ as
\begin{equation}
    \label{eq:Psi}
    \Psi^{(j)} (x) := \varphi_{\sigma^{(j)} (x)}^{(j+1)} \circ R \circ h^{(j)} (x).
\end{equation}
As shown in Fig 2, $\Psi^{(j)}$ first causes $x$ to flow up to $G^{(j)}$ along $F^{(j)}$ over the time $\sigma^{(j)}$, resets to $R(G^{(j)}) \in M^{(j+1)}$, and flows along $F^{(j+1)}$ for the time $\sigma^{(j)}$. 
\begin{figure}
    \centering
    \includegraphics[width=0.95\linewidth]{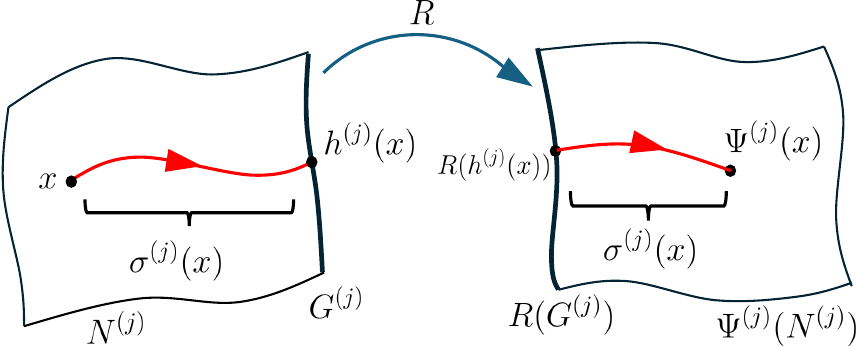}
    \caption{Construction of a map $\Psi$.}
    \label{fig:Psi}
\end{figure}
With the appropriate introduction of the collar neighborhood $N^{(j)}$, this map is well-defined, and is a diffeomorphism from Assumptions \ref{ass:transverseG} and \ref{ass:transverseRG}. 
Let $N:= \amalg_{j\in J} N^{(j)}\subset M$ and define $\Psi:N\to \Psi(N)$ such that $\Psi|_{N^{(j)}}=\Psi^{(j)}$. 
The derivative ${\rm D}\Psi$ acts as a pushforward that maps a vector field on $N$ to a vector field on $\Psi(N)$. 
Recall generally that pushforward is defined as the derivative of a diffeomorphism such that $({\rm D} \Psi F') (x):= {\rm D} \Psi (\Psi^{-1} (x)) F'(\Psi^{-1} (x))$, where $F':N\to {\rm T}N$ is a vector field on $N$, $x\in \Psi (N)$ is a point, and ${\rm D} \Psi(\Psi^{-1} (x))$ is a linear map that sends a tangent vector on ${\rm T}_{\Psi^{-1}(x)} N$ to a tangent space ${\rm T}_{x} \Psi(N)$, analogous to Jacobian (if one chooses bases of ${\rm T}_{\Psi^{-1}(x)} N$ and ${\rm T}_{x} \Psi(N)$, one can obtain the matrix representation of ${\rm D} \Psi(\Psi^{-1} (x))$). 
From Assumptions \ref{ass:transverseG} and \ref{ass:transverseRG}, ${\rm rank}\, {\rm D}\Psi (x) = n$ for all $x\in N$ holds. 
The following lemma is essential to derive the observable space. 
\begin{lemma}
\label{lem:existence_of_vector}
Under Assumptions \ref{ass:transverseG}-\ref{ass:GESHLC}, 
there exists a $n-1$-tuple $C^r$ vector fields $F_2, \ldots, F_n: N\to {\rm T} N$ such that
\begin{enumerate}[label=(\alph*)]
\item $\forall z \in G$, ${\rm T}_z G = {\rm span}\left\{ F_2(z), \ldots, F_n(z) \right\}$
\item Each pair of $F, F_2, \ldots, F_n $ is commutative on $N$.
\end{enumerate}
\end{lemma}
Now, we are in a position to define the observable space $\mathcal{F}$ which is equipped with $k$-times differentiability ($0\leq k \leq r$). 
This space will be denoted as $\tilde{C}^k (M)$ in the following.
\begin{definition}
\label{def:tildeC}
Under Assumptions \ref{ass:transverseG}-\ref{ass:GESHLC}, 
we define $\tilde{C}^k (M)$ as an observable space such that for all $f\in \tilde{C}^k (M)$, $f|_{M^{(j)}}\in C^{k}(M^{(j)})$ and 
\begin{equation}
    \label{eq:lie}
    \begin{aligned}
        &\forall z \in G, ~\forall (l_1,\ldots,l_n) \in (\mathbb{N}_{\geq 0})^n ~{\rm s.t.}~l_1 + \ldots + l_n \leq k, \\
    & L_{F}^{l_1} L_{F_2}^{l_2} \cdots  L_{F_n}^{l_n} f (z) = L_{F}^{l_1} L_{{\rm D}\Psi F_2}^{l_2} \cdots  L_{{\rm D}\Psi F_n}^{l_n} f (R(z)),
    \end{aligned}
\end{equation}
where $\mathbb{N}_{\geq 0}$ are non-negative integers and $L_F$ is a Lie derivative with respect to a vector field $F$.
\end{definition}
One can verify that $\tilde{C}^k(M)$ is a vector space and $k < k' \Longrightarrow \tilde{C}^k (M) \subset \tilde{C}^{k'} (M)$. 
The following theorem is one of the main results (its proof is shown in Appendix \ref{app:proofs}). 
\begin{theorem}
    \label{thm:smooth}
    Let $H$ be a hybrid dynamical systems satisfying Assumptions \ref{ass:transverseG}-\ref{ass:GESHLC} and let $\varphi_t$ be its flow.  
    Consider the smooth dynamical system $(\tilde{M},\tilde{\varphi}_t)$ linked by Eq. \eqref{eq:conjugacy}. 
    Then, the following equivalence holds:
    \begin{equation}
        \label{eq:iff}
        f\in \tilde{C}^k (M) \iff f \circ \pi^{-1} \in C^k(\tilde{M}). 
    \end{equation}
    Furthermore, $\tilde{C}^k (M)$ is invariant under the action of $U_t$:
    \begin{equation}
        \label{eq:invariance}
        f \in {\tilde{C}}^k (M) \Longrightarrow U^t f \in {\tilde{C}}^k (M),~\forall t\in [0,\infty).
    \end{equation}
\end{theorem}
\begin{remark}
    \label{rem:projection}
First, we note that $f\circ \pi^{-1}$ of Eq. \eqref{eq:iff} is well-defined. 
Consider the case when $k=0$. 
In this case, the condition \eqref{eq:lie} is equivalent to
\begin{equation}
    \label{eq:lie0}
    \forall y \in G,\quad f(y) = f(R(y)).
\end{equation}
The \emph{fiber} of $\pi$, denoted by $\pi^{-1}:\tilde{M}\to M$, is given by
\begin{equation}
\label{eq:fiber}
    \pi^{-1}(\tilde{x}) = \left\{ \begin{alignedat}{2}
    &\{ x\in M ~|~ \pi(x) = \tilde{x} \},\quad& & if~ \tilde{x} \in \tilde{M}\backslash \pi (G) \\
    &\{ z\in G,~R(z) \in R(G) ~|~ & & \pi  (z) = \tilde{x} \}, \\ 
    & \quad& &if~ \tilde{x} \in \tilde{M} \cap \pi (G).
\end{alignedat} \right.
\end{equation}
The condition \eqref{eq:lie0} ensures that $f$ is constant on the fiber of $\pi$. 
Additionally, by Definition \ref{def:tildeC}, $f \in \tilde{C}^0(M)$ must satisfy $f|_{M^{(j)}} \in {C}^0 (M^{(j)})$ for all $j\in J$.
Then, there exists a function $g\in C^0 (\tilde{M})$ such that $g=f\circ \pi $, uniquely determined by \cite[Theorem A.30]{Lee_smooth}. 
In Eq. \eqref{eq:iff}, this $g$ is represented as $f\circ \pi^{-1}$. 
\end{remark}

Here, we outline the meanings of (a) and (b) of Lemma \ref{lem:existence_of_vector}, Definition \ref{def:tildeC}, and Theorem \ref{thm:smooth}, using the key mathematical concepts from smooth manifold theory. 
The reason why $n-1$-tuple vector fields $F_2,\ldots,F_n$ satisfying (a) and  (b) of Lemma \ref{lem:existence_of_vector} are introduced is to provide a \emph{chart} of \emph{atlas} covering $\tilde{M}$ in Lemma \ref{lem:smoothing}. 
If $n$-tuple vector fields $F,F_2,\ldots,F_n$ commute and their tangent vectors at $x \in M$ span the tangent space ${\rm T}_x M$, then $(F,F_2,\ldots,F_n)$ is a \emph{coordinate frame} and these vector fields form a chart \cite[Theorem 9.46]{Lee_smooth}. 
Moreover, $({\rm D}\Psi F, \ldots, {\rm D}\Psi F_n)$ is also a coordinate frame because the pushforward map preserves both the linear independence (a) and commutative property (b), thereby form charts.
These two charts have boundaries corresponding to $G$ and $R(G)$, and glue their boundaries together through the gluing map $\pi$.
Actually, smoothness on the gluing surface $\pi(G)$ persists through the gluing process. 
In other words, we have defined $\Psi$ to make the vicinity of $\pi(G)$ smooth. 
To verify the differentiability of $f\circ \pi$ on $\tilde{M}$, based on the basic definition of a differentialable function on the manifold, it is enough to verify the differentiability of $f\circ \pi$ via the charts. 
As mentioned, the charts around $\pi(G)$ are formed by the coordinate frames $(F,F_2,\ldots,F_n)$, allowing one to calculate the derivatives through Lie derivatives of $F,F_2,\ldots,F_n$. 
Eq. \eqref{eq:lie} of Definition \ref{def:tildeC} ensures the continuity of (high-order) derivatives through the charts of coordinate frames.
This is the outline of the proof of Theorem \ref{thm:smooth}.

\section{
Spectral Properties
}
\label{sec:KEF}
In this section we explore the spectrum of the Koopman operator for the hybrid system $H$. 
For a smooth dynamical system with an asymptotically stable limit cycle, Koopman eigenvalues are determined using the eigenvalues of ${\rm D}P(x^*)$, where $P$ is the associated Poincare map and $x^*$ is the fixed point. 
As described in \cite{kvalheim2021existence}, consider the flow to be $C^r$ and an observable space also to be $C^r$. 
Suppose that the system has an asymptotically stable limit cycle with period $\tau$, and that the associated Poincare map has $n-1$ eigenvalues $\rho_2,\ldots,\rho_n$ such that $1>|\rho_2|\geq \ldots \geq |\rho_n|$, called Floquet multiplicities. 
Let $\omega= \tau/2\pi$ be a angular fundamental frequency of the limit cycle and $\nu_2:=\ln ({\rho_2})/\tau,\ldots, \nu_n:=\ln ({\rho_n})/\tau$ be the Floquet exponents.
Under $r$-nonresonant and spectral spread conditions\footnote{
For all non-negative integers $m_2,\ldots,m_n$ such that $2\leq m_2+\ldots + m_n <r+1$, an $r$-nonresonant condition is described as $\nu_i \neq m_2 \nu_2 + \cdots + m_n \nu_n,\quad i=2,\ldots,n$. 
Also, a spectral spread condition is described $|\rho_n| > |\rho_2|^r$, or equivalently, ${\rm Re} \, \nu_n > r {\rm Re} \, \nu_2$.
}, 
${\rm i}\omega$ and $\nu_2,\ldots, ~\nu_n $ are Koopman eigenvalues. 
Additionally, the associated Koopman eigenfunctions $\phi_{{\rm i}\omega},~\phi_{\nu_2},\ldots, \phi_{\nu_n} \in C^r$, called Koopman principal eigenfunctions, exist uniquely. 

Here, we extend the results of \cite{kvalheim2021existence} to the case of hybrid limit cycle. 
The following theorem is our second result, which ensures the existence and uniqueness of the Koopman eigenfunctions (its proof is in Appendix \ref{app:proofs}). 
\begin{theorem}
    \label{thm:existence_KEF}
    Consider $H$ to be a hybrid dynamical system satisfying Assumptions \ref{ass:transverseG}-\ref{ass:GESHLC}. 
    The fundamental frequency $\omega$ and Floquet exponents $\nu_2, \ldots, \nu_n$ satisfy $r$-nonresonant and spectral spread conditions. 
    Then, there exist Koopman eigenfunctions $\phi_{{\rm i}\omega},~\phi_{\nu_2},\ldots, \phi_{\nu_n}$ uniquely such that 
    \begin{equation}
    \label{eq:KEFomega}
        U_t \phi_{{\rm i}\omega} = {\rm e}^{{\rm i}\omega t} \phi_{{\rm i}\omega},\quad \phi_{{\rm i}\omega}\in \tilde{C}^r(M)\backslash \{ 0 \},
    \end{equation}
    \begin{equation}
    \label{eq:KEFnu}
        U_t \phi_{\nu_i} = {\rm e}^{{\nu_i} t} \phi_{\nu_i},\quad \phi_{\nu_j}\in \tilde{C}^r(M)\backslash \{ 0 \}.
    \end{equation}
\end{theorem}
This theorem is closely related to the existence of \emph{linear embeddings} \cite{liu2023non,kvalheim2023linearizability}, which can be identified through data-driven methods \cite{kutz2016dynamic}. 
For a continuous dynamical system $({M},\varphi_t)$, a linear embedding is defined as a continuous map $E:M\hookrightarrow \mathbb{R}^m~(m\geq {\rm dim}\,M)$ such that $E \circ \varphi_t (x)=  {\rm e}^{At} E(x)$ for all $x\in M$, with a matrix $A \in \mathbb{R}^{m\times m}$. 
Through a quotient topology \cite{simic2005towards}, this notion is extended to the hybrid limit-cycling systems via the existence of Koopman principal eigenfunctions (its proof is straightforward from Theorem 2, and so omitted). 
\begin{corollary}
    \label{col:embeddings}
    Consider $H$ to be a hybrid dynamical system satisfying Assumptions \ref{ass:transverseG}-\ref{ass:GESHLC}. 
    There exists an embedding $E:M \hookrightarrow \mathbb{R}^m ~ (m\geq n)$ such that $E \circ \varphi_t (x) ={\rm e}^{At} E (x)$. 
\end{corollary}
One can interpret this corollary as indicating that $\tilde{M}$ of Lemma \ref{lem:smoothing} is linearly embedded by a linear immersion $\pi$ into a high-dimensional Euclidean space. 

\section{
An Example
}
\label{sec:example}
Consider the following hybrid dynamical system $H= (J,M,F,G,R)$:
\begin{itemize}
    \item $J=\{1\}$
    \item $M=\{ (x_1,x_2) \in\mathbb{R}^2 ~|~ x_1 \in [1,2],~x_2>0 \}$
    \item $F(x_1,x_2) = -x_1 \left(\frac{\partial}{\partial x_1}\right) - 2x_2 \left(\frac{\partial}{\partial x_2}\right)$
    \item $G=\{ (x_1,x_2) \in \mathbb{R}^2 ~|~ x_1 = 1,~x_2 >0 \}$
    \item $R(1,x_2) = (2,x_2+1)$
\end{itemize}
This system possesses the globally asymptotically stable hybrid limit cycle $\Gamma = \{ (x_1, x_2)\in M ~|~ x_2 - x_1^2/3 = 0 \}$, as shown in Fig. \ref{fig:phase_portrait}. 
It can also be seen that the period of $\Gamma$ is $\tau = \ln 2$, angular frequency is $\omega = 2\pi/\ln 2$, the Floquet multiplicity is $\rho = 1/4$, and the Floquet exponent is $\nu = -2$.
First, we derive $\Psi$ from Eq. \eqref{eq:Psi}. 
\begin{figure*}[t]
  \begin{tabular}{ccc}
    \begin{minipage}[t]{0.33\hsize}
      \centering
      \includegraphics[keepaspectratio, width=1\linewidth]{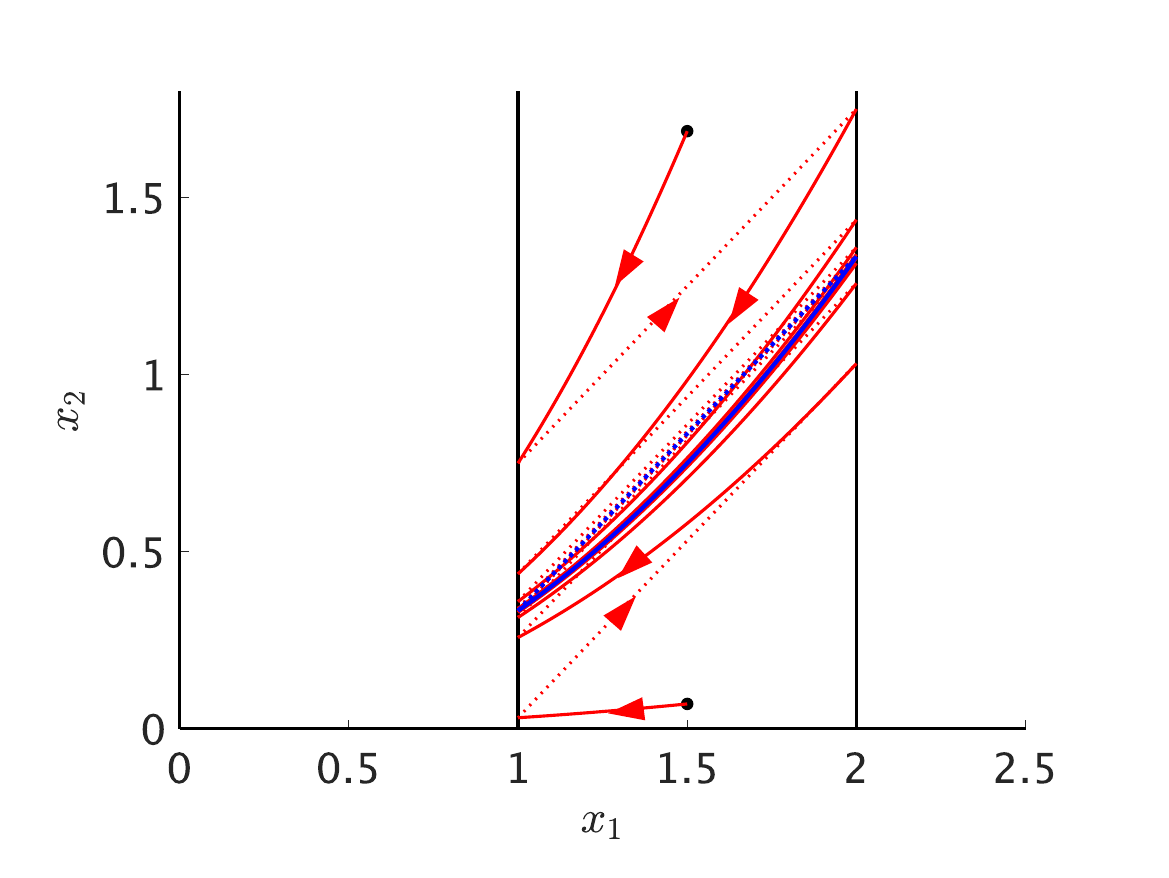}
      \subcaption{Phase portrait of $H$. Red curve represents the integral curve of $F$, red dashed line the reset $R$, blue curve the limit cycle, and blue dashed curve the reset in terms of the limit cycle.}
      \label{fig:phase_portrait}
    \end{minipage} &
    \begin{minipage}[t]{0.3\hsize}
      \centering
      \includegraphics[keepaspectratio, width=1\linewidth]{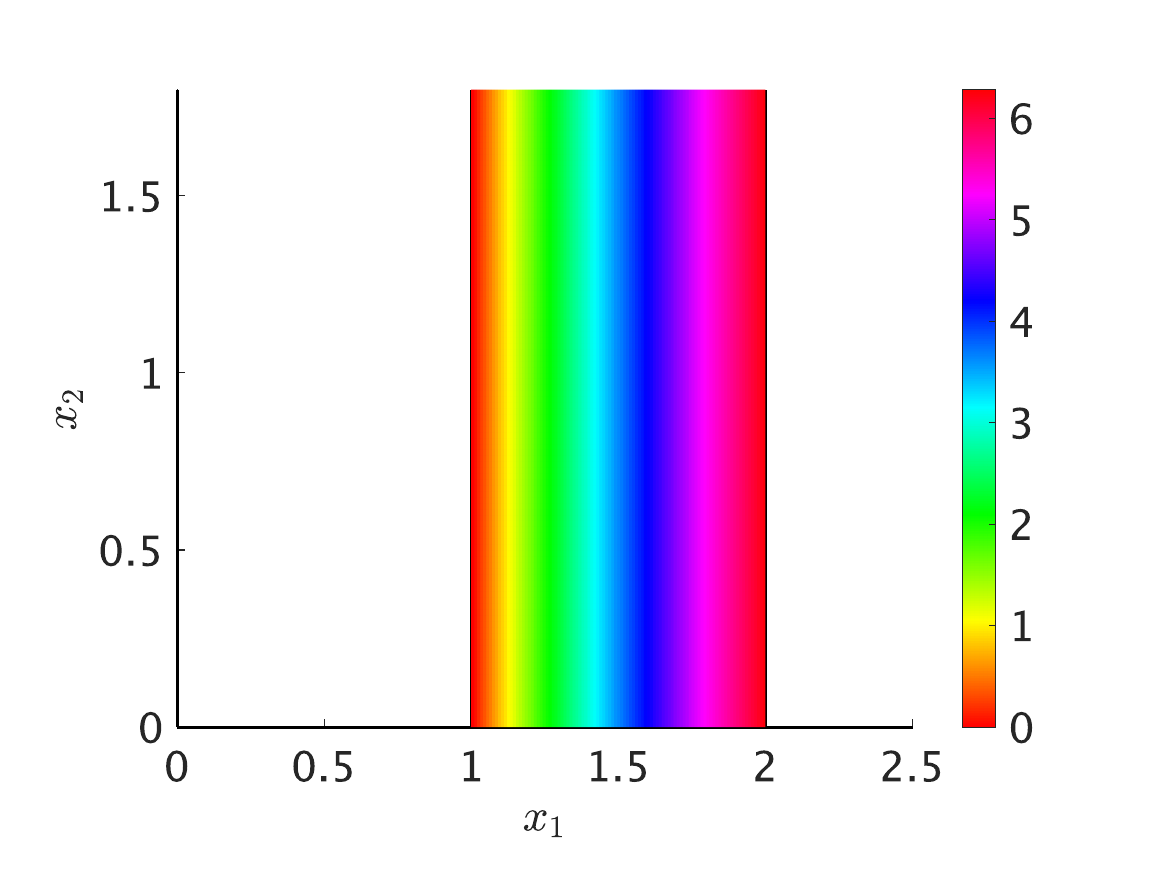}
      \subcaption{The argument of Koopman eigenfunction $\angle \phi_{{\rm i}\omega}$.}
      \label{fig:phi_iomega}
    \end{minipage} &
    \begin{minipage}[t]{0.3\hsize}
      \centering
      \includegraphics[keepaspectratio, width=1\linewidth]{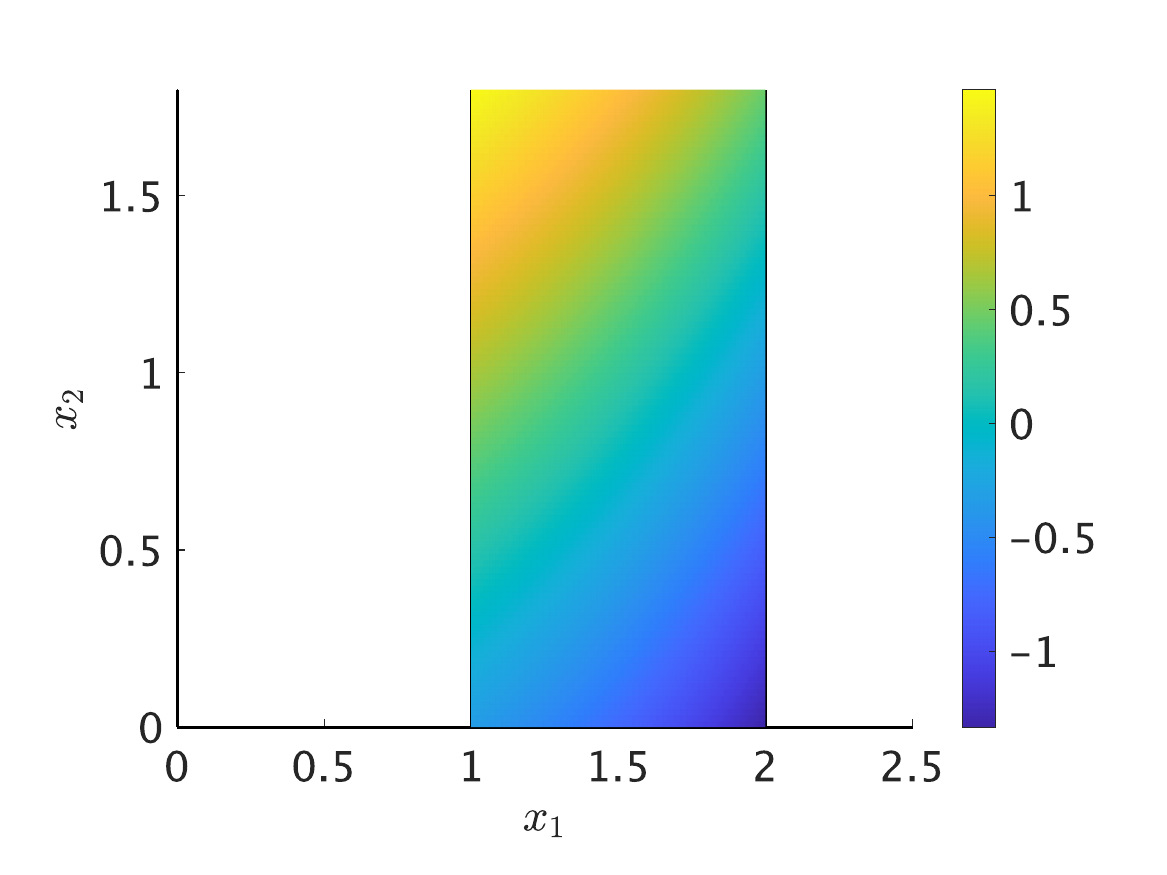}
      \subcaption{Koopman eigenfunction $\phi_{\nu}$.}
      \label{fig:phi_nu}
    \end{minipage}
  \end{tabular}
  \caption{An illustrative example.}
\end{figure*}
A local flow and time-to-impact map can be calculated as $\varphi_t^{(1)}(x_1, x_2) = (x_1 {\rm e}^{-t}, x_2 {\rm e}^{-2t})$ and $\sigma (x_1,x_2) = \ln x_1$, along with $h(x) = (1,x_2/x_1^2)$.
From this, we obtain $\Psi (x_1,x_2)= (2/x_1, x_2/x_1^4)$.
Let $F_2 (x_1, x_2) := x_2 \left( {\partial}/{\partial x_2}\right)$, that commutes with $F$. 
The pushforward of $F_2$ by $\Psi$ can be calculated $
( {\rm D}\Psi F_2 )(x_1, x_2) = {\rm D}\Psi (\Psi^{-1} (x_1,x_2)) F_2 (\Psi^{-1} (x_1,x_2)) = \left(x_2 - {x_1^2}/{4} \right)\left({\partial}/{\partial x_2}\right)$. 
Thus, we can endow $\tilde{C}^k(M)~(k=0,1,\ldots)$. 
For example, the condition \eqref{eq:lie} where $k=1$ is such that for all $x_2 > 0$, 
$$
~\quad f(1,x_2) = f(2,x_2+1),
$$
{\footnotesize$$
\begin{aligned}
    \left.\left( -x_1 \frac{\partial f}{\partial x_1} - 2 x_2 \frac{\partial f}{\partial x_2} \right)\right|_{\substack{x_1 = 1 \\ x_2 = x_2}} = \left.\left( -x_1 \frac{\partial f}{\partial x_1} - 2 x_2 \frac{\partial f}{\partial x_2} \right)\right|_{\substack{x_1 = 2 \\ x_2 = x_2 + 1},}
\end{aligned}
$$}
{$$
\qquad \qquad \quad ~ \left.x_2 \frac{\partial f}{\partial x_2} \right|_{\substack{x_1 = 1}} =  \left.\left(x_2 -\frac{x_1^2}{4}\right)\frac{\partial f}{\partial x_2}\right|_{\substack{x_1 = 2 \\ x_2 = x_2 + 1}}.
$$}
One can check that $\phi_{{\rm i} \omega} (x_1,x_2) := {\rm e}^{{\rm i}2\pi \frac{\ln x_1}{\ln 2}}$ and $\phi_{\nu} (x_1,x_2) :=  x_2 - x_1^2/3$ belong to $\tilde{C}^1(M)$. 
In fact, these observables are Koopman eigenfunctions that belong to $\tilde{C}^\infty(M)$, as shown in Figs. \ref{fig:phi_iomega} and \ref{fig:phi_nu}. 

\section{Conclusion} 
\label{sec:conclusion}
In this paper, we extended the Koopman operator theory developed in smooth dynamical systems to hybrid dynamical systems with hybrid limit cycles. 
To maintain the smooth structure intrinsic to the hybrid systems, we characterized the observable space required for this extension. 
On the defined observable space, we showed the existence and uniqueness of Koopman eigenfunctions. 
Future work includes extending our results to hybrid systems with other types of $\omega$-limit sets. 
Additionally, the results may be applied to the analysis of singularly-perturbed systems \cite{katayama2024koopman}.

\section*{Acknowledgment}

The authors appreciate Professors Alexandre Mauroy and Kazuma Sekiguchi for their valuable comments of the manuscript.

\appendices
\section{Definition of Hybrid Limit Cycles}
\label{app:definition}
Here, we provide a precise definition of hybrid limit cycles. 
Consider a periodic orbit $\Gamma = \{ \varphi_t \}_{t\in [0,\tau)}$ with period $\tau$, where $\varphi_t$ is the flow of the hybrid system $H = (J,M,F,G,R)$ introduced in Section \ref{sec:Hybrid_intro}. 
Due to its periodicity, one can consider the Poincaré map analogously to smooth dynamical systems theory. 
Recall that an associated Poincaré section need to be a codimension-1 manifold that is transversal to the vector field $F$. 
Our definition of the globally asymptotically stable hybrid limit cycle is based on the Poincare map and is equivalent to Lyapunov stability \cite{grizzle2001asymptotically,morris2005restricted}. 
\begin{definition}
\label{def:GEHLC}
    Let $\Sigma \subset M^{(j)}$ be an appropriately chosen Poincaré section and $P_{\Sigma}$ be an associated Poincare map. 
    A \emph{globally asymptotically stable hybrid limit cycle} is defined as a set $\Gamma$ such that for any $j\in J$ and $x \in M^{(j)}$, the trajectory of $P|_{\Sigma}$ from $x$ converges to a fixed point $x^* \in \Gamma \cup \Sigma$. 
\end{definition}
Note that since $F$ is smooth and $R$ is a diffeomorphism, $P|_{\Sigma}$ is smooth, and ${\rm D}P_{\Sigma} (x^*)$ has constant eigenvalues for all choice of $\Sigma$, where ${\rm D}P_{\Sigma} (x^*)$ represents the derivative of $P_{\Sigma}$ at $x^*$.

\section{Proofs}
\label{app:proofs}
{
\subsection{Preliminaries}
Here, we review the basic definitions and a proposition related to smooth manifold theory with reference to \cite{Lee_smooth}. 
We begin by summarizing the notations for an $n$-dimensional $C^r$ manifold $M$. 
An atlas of $M$ is denoted as $\{(U_\alpha,\vartheta_\alpha)\}_{\alpha \in A}$, where $A$ is a set of indices. 
We denote the coordinate map as $\vartheta_\alpha=(\vartheta_\alpha^1, \ldots,\vartheta_\alpha^n)$ and local coordinate as $y^i = \vartheta_\alpha^i (x)$ with $x\in M$. 
The left and right half spaces are defined as $\mathbb{H}_-^n := \{ (y^1, \ldots, y^n) \in \mathbb{R}^n ~|~ y^1 \leq 0 \}$ and $\mathbb{H}_+^n := \{ (y^1, \ldots, y^n) \in \mathbb{R}^n ~|~ y^1 \geq 0 \}$. 
At $x\in M$, a vector field $F:M\to {\rm T}M$ is denoted as $F_x$, the derivative of a diffeomorphism $h:M\to h(M)$ denoted as ${\rm D}h_x$, and the differential form of a function $f:M\to \mathbb{C}$ denoted as ${\rm d}f_x$. 
A pushforward of $F$ by $h$ is denoted as $({\rm D}h F)_{p} = {\rm D}h_{h^{-1} (p)} F_{h^{-1} (p)}$, where $p=h(x)$ (Note that the notation differs from that of the main text). 
The derivative of $f$ with respect to $F$ at $x \in M$ is denoted as $Ff_x$, $L_F f_x$, or ${\rm d}f_x (F)$. 
One can calculate the derivative of $f\circ h$ with respect to $F$ as follows: for $x\in M$
\begin{equation}
    \label{eq:pullback}
    \begin{aligned}
        {\rm d}(f\circ h)_x (F)  = ({\rm D}^* h {\rm d}f)_x (F)= {\rm d}f_{h(x)} ( {\rm D}h F),
    \end{aligned}
\end{equation}
where ${\rm D}^*h$ is the dual operator of ${\rm D}h$, known as a pullback. 

Consider a $n$-dimensional manifold $M$ endowed with an atlas $\{ (U_\alpha,\vartheta_\alpha) \}_{\alpha\in A}$. 
Here, we review the definitions of $C^r$ structures, $C^r$ functions, and $C^r$ vector fields.
\begin{definition}
    \label{def:smooth_structure}
    If for any two charts $(U_\alpha, \vartheta_\alpha)$ and $(U_{\alpha'}, \vartheta_{\alpha'})$ such that $U_\alpha \cap U_{\alpha'} \neq \emptyset$ and the transition map $\vartheta_\alpha \circ \vartheta_{\alpha'}^{-1}$ is $C^r$ diffeomorphic, then $M$ is said to has $C^r$ structures and is said to be a $C^r$ manifold. 
\end{definition}
\begin{definition}
    \label{def:smooth_functions}
    A function $f:M\to \mathbb{C}$ is $C^r$ if for any chart $(U_\alpha, \vartheta_\alpha)$, $f\circ \vartheta_\alpha^{-1}$ is $C^r$. 
\end{definition}
\begin{definition}{(see also \cite[Proposition 8.1]{Lee_smooth})}
    \label{def:smooth_vectorfield}
    If a vector field $F$ is expanded as
    $$F(x) = F^1 (x) \left( \frac{\partial}{\partial y^1} \right)_x + \ldots + F^n (x) \left( \frac{\partial}{\partial y^n} \right)_x, $$
    where $\frac{\partial}{\partial y^1}, \ldots, \frac{\partial}{\partial y^n}$ are coordinate vector fields in terms of a given chart, $F^1,\ldots,F^n:M\to \mathbb{R}$ are called \emph{component functions}. 
    Moreover, a vector field $F:M\to {\rm T}M$ is $C^r$ if, for any chart $(U_\alpha, \vartheta_\alpha)$, all of the component functions are $C^r$. 
\end{definition}
At the end of the review, the following proposition is useful to verify the differentiability of the vector field.
\begin{proposition}
    \label{prop:smooth_vectorfield}
    A component function $F^i$ in terms of a chart $(U_\alpha, \vartheta_\alpha)$ is $C^r$ if and only if ${\rm d}\vartheta_\alpha^i|_{\vartheta_\alpha^{-1} (y)} (F)$ is $C^r$ with respect to $y \in \mathbb{R}^n$. 
\end{proposition}

Before proving main theorems, we introduce some important lemmas. 
Recall collar neighborhoods $N^{(j)}$, time-to-impact maps $\sigma^{(j)}$ and projections $h^{(j)}$ are introduced from Section \ref{sec:Hybrid_intro}, and $\Psi$ in Section \ref{sec:Koopman}. 
\begin{lemma}
    \label{lem:notation}
    \begin{enumerate}[label=(\roman*)]
        \item For any $z\in G^{(j)}$, suppose $t\leq 0$ is such that $\varphi_{t}^{(j)} (z) \in N^{(j)}$. Then the following holds;
        \begin{equation}
        \label{eq:lem1}
            ~~\sigma^{(j)} \circ \varphi_t^{(j)} (z) = -t,
        \end{equation}
        \begin{equation}
            \label{eq:lem2}
            h^{(j)} \circ \varphi_t^{(j)} (z) = z.
        \end{equation}
        \item For any $x\in N^{(j)}$, suppose $t\in\mathbb{R}$ is such that $\varphi_{-t}^{(j)} (x) \in N^{(j)}$. Then the following holds
        \begin{equation}
            \label{eq:lem11}
            ~~\quad \sigma^{(j)} \circ \varphi_{-t}^{(j)} (x) = t + \sigma^{(j)} (x),
        \end{equation}
        \begin{equation}
            \label{eq:lem22}
            h^{(j)} \circ \varphi_{-t}^{(j)} (x) = h^{(j)} (x)
        \end{equation}
        \begin{equation}
        \label{eq:lem3}
            \qquad \varphi_t^{(j+1)} \circ \Psi (x) =\Psi \circ  \varphi_{-t}^{(j)} (x).
        \end{equation}
    \end{enumerate}
\end{lemma}
\begin{proof}
    (i) Eq. \eqref{eq:lem1} follows directly from the definition of $\sigma^{(j)}$. Eq. \eqref{eq:lem2} is also straightforward from the definition $h^{(j)}(x) := \varphi_{\sigma^{(j)}(x)}^{(j)} (x)$. 
    (ii) For any $x\in N^{(j)}$, substituting $z=\varphi_{\sigma^{(j)} (x)} (x)$ and $t = -t'-\sigma^{(j)} (x)$ into Eq. \eqref{eq:lem1}, we can calculate $\sigma^{(j)}\circ \varphi_{-t'-\sigma^{(j)}(x)}^{(j)} (\varphi_{\sigma^{(j)}(x)} (x)) =\sigma^{(j)} \circ \varphi_{-t'}^{(j)} (x) =  t'+\sigma^{(j)} (x)$, which gives Eq. \eqref{eq:lem11}. 
    Similarly, substituting them into Eq. \eqref{eq:lem2}, we can calculate $h^{(j)} \circ \varphi_{-t'-\sigma^{(j)}(x)}^{(j)} (\varphi_{\sigma^{(j)}(x)} (x)) = h^{(j)} \circ \varphi_{-t'}^{(j)} (x) = \varphi_{\sigma^{(j)}(x)} (x) = h^{(j)}(x)$, which gives Eq. \eqref{eq:lem22}.
    For Eq. \eqref{eq:lem3}, by the definition of $\Psi$ in \eqref{eq:Psi}, the right-hand-side of Eq. \eqref{eq:lem3} is $\varphi_{\sigma^{(j)}(\varphi_{-t}^{(j)} (x))}^{(j+1)} \circ R \circ h^{(j)}\circ \varphi_{-t}^{(j)} (x)$. 
    Using Eqs \eqref{eq:lem11} and \eqref{eq:lem22}, this becomes $ \varphi_{t+\sigma^{(j)}(x)}^{(j+1)} \circ R \circ h^{(j)} (x) $. 
    The left-hand-side of Eq. \eqref{eq:lem3} is $\varphi_t^{(j+1)} \circ \varphi_{\sigma^{(j)}(x)}^{(j+1)} \circ R \circ h^{(j)} (x)$ by the definition \eqref{eq:Psi}, which corresponds to the right-hand-side one of Eq. \eqref{eq:lem3}. Thus, Eq. \eqref{eq:lem3} holds.
\end{proof}
\begin{corollary}
    \label{col:dif_Psi}
    For any $x\in N^{(j)}$,
        \begin{equation}
        \label{eq:time-to-impact-dif}
        \qquad \, {\rm d} \sigma^{(j)}_x F = -1,
        \end{equation}
        \begin{equation}
            \label{eq:hF}
            ({\rm D}h^{(j)} F)_x = 0,
        \end{equation}
        \begin{equation}
        \label{eq:Psi_inv_dif}
        \quad ~ ({\rm D}\Psi^{-1} F)_x = -F_x.
        \end{equation}
\end{corollary}
\begin{proof}
    Eq. \eqref{eq:time-to-impact-dif} (and \eqref{eq:hF}) is directly obtained by differentiating Eq. \eqref{eq:lem11} (and \eqref{eq:lem22}) with respect to $t$ and substituting $t=0$ (see Proposition 3.24 and 11.23 of \cite{Lee_smooth}). 
    To show Eq. \eqref{eq:Psi_inv_dif}, because $\Psi|_{N^{(j)}}$ is a diffeomorphism, $\Psi^{-1} \circ \varphi_t^{(j+1)} (x') = \varphi_{-t}^{(j)} \circ \Psi^{-1} (x')$ holds for any $x'\in \Psi(N^{(j)})$. 
    Then, differentiating this with respect to $t$ gives Eq. \eqref{eq:Psi_inv_dif}. 
\end{proof}

\subsection{Proof of Lemma 1}
\begin{proof}
The proof consists of 3 steps; (i) constituting $\tilde{M}$ and $\pi$; (ii) endowing $\tilde{M}$ with a $C^r$ structure and showing $\pi|_{M^{(j)}}$ is a diffeomorphism; (iii) showing $\tilde{\varphi}_t$ is $C^r$. 

(i) Because $R$ is a homeomorphism, an equivalence relation $G\sim R(G)$ can be endowed. According to \cite{simic2005towards,burden2015model}, endowing $M/\sim$ with a quotient topology, the quotient space $\tilde{M} := M/\sim$ becomes an $n$-dimensional topological manifold. 
Let $\pi:M\to M/\sim$ be a canonical projection to the equivalent class. 
Then, $\pi|_{M^{(j)}}: M^{(j)} \to \pi(M^{(j)})$ is a homeomorphism.

(ii) 
Here, we show $\tilde{M}$ possesses a $C^r$ structure with respect to an appropriate atlas. 
First, since $M^{(j)}\backslash (G^{(j)}\cup R(G^{(j-1)}))$ is a $C^r$ manifold with or without boundary, there is an atlas $\{(U_\alpha,\vartheta_\alpha)\}_{\alpha\in A^{(j)}}$ covering $M^{(j)}\backslash (G^{(j)}\cup R(G^{(j-1)}))$, endowing it with a $C^r$ structure. 
One can verify that $\{(\tilde{U}_\alpha, \tilde{\vartheta}_\alpha)\}_{\alpha\in A^{(j)}, j\in J}$, where $\tilde{U}_\alpha:=\pi(U_\alpha)$ and $\tilde{\vartheta}_\alpha = \vartheta_\alpha\circ \pi^{-1})$, covers $\tilde{M}$ except the neighborhood of $\pi(G)$ (``the seams of gluing''), and that their transition maps between these charts are $C^r$. 
It is straightforward that $\pi|_{M^{(j)}}: M^{(j)} \to \pi(M^{(j)})$ is a $C^r$ diffeomorphism, which shows one of the proofs at (ii). 
Our approach to completing the proof of (ii) is to construct charts $(\tilde{V}_\beta, \tilde{\eta}_\beta)$ for $\beta \in B^{(j)}$ that cover the neighborhood of $\pi (G)$, and to show that three transition maps between $\alpha\in A^{(j)}$ and $\beta \in B^{(j)}$, between $\alpha\in A^{(j+1)}$ and $\beta \in B^{(j)}$, and between $\beta\in B^{(j)}$ and $\beta' \in B^{(j)}$, are $C^r$. 
For each $j\in J$, we introduce a $C^r$ atlas on each $(n-1)$-dimensional manifold $G^{(j)}$, denoted by $\{ (V_\beta, \zeta_\beta)\}_{\beta \in B^{(j)}}$, where $\zeta_\beta = (\zeta_\beta^2,\ldots,\zeta_\beta^n)$. 
Define ${ (h^{(j)}) }^{-1} (V_\beta)$ as $${ (h^{(j)}) }^{-1} (V_\beta) := \{ x\in N^{(j)} ~|~ h^{(j)} (x) \in V_\beta \}.$$
This is diffeomorphic to a subset of $\mathbb{H}_-^n$ \cite[Theorem 2.94]{Lee_smooth}, and the collection $\{ {(h^{(j)})}^{-1} (V_\beta) \}_{\beta\in B^{(j)}}$ covers the collar neighborhood $N^{(j)}$. 
Analogous to the collar neighborhood $N^{(j)}$ around $G^{(j)}$, $\Psi (N^{(j)})$ forms a collar neighborhood of $R(G^{(j)}) \subset \partial M^{(j+1)}$. 
Then, $\Psi \left( { (h^{(j)}) }^{-1} (V_\beta) \right)$ is diffeomorphic to a subset of $\mathbb{H}_+^n$ and covers the collar neighborhood $\Psi (N^{(j)})$ of $R(G^{(j)})$. 
\begin{figure}
    \centering
    \includegraphics[width=0.9\linewidth]{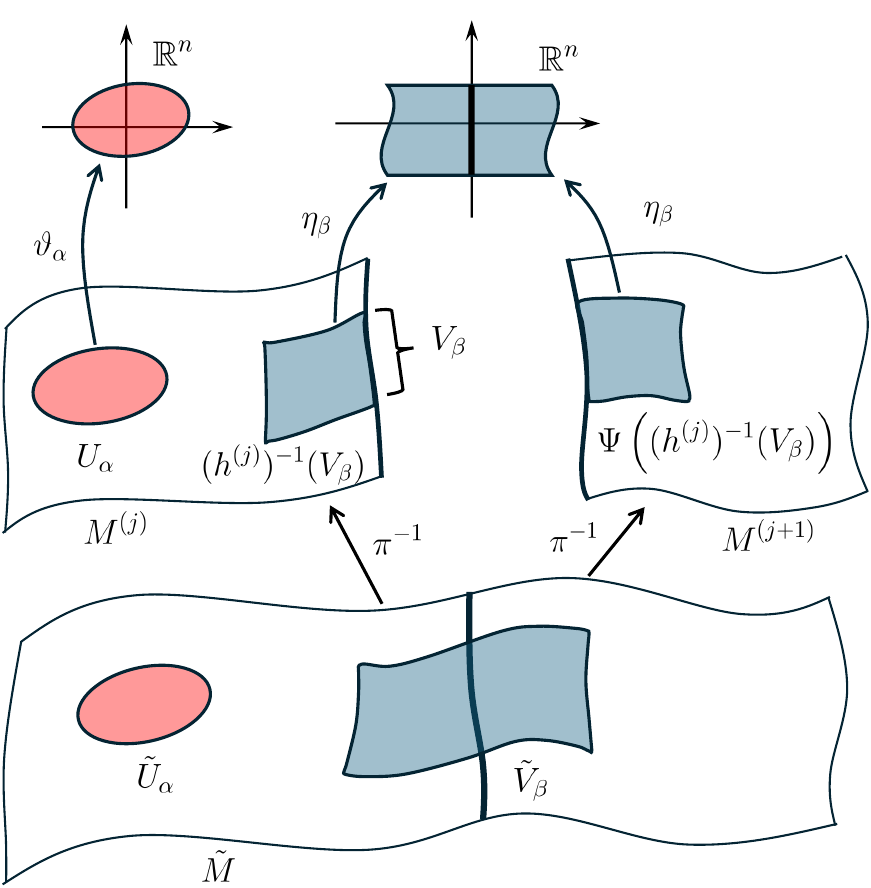}
    \caption{Constructions of the atlas \eqref{eq:atlas}. }
    \label{fig:atlas}
\end{figure}
Consequently, $\tilde{V}_\beta := \pi \left( {(h^{(j)})}^{-1} (V_\beta) \right) \cup \pi \left(   \Psi \left( {(h^{(j)})}^{-1} (V_\beta) \right) \right)$
is diffeomorphic to a subset of $\mathbb{R}^n$ as shown in Fig. \ref{fig:atlas}, and $\{ \tilde{V}_\beta \}_{\beta \in B^{(j)}}$ covers $\pi(N^{(j)} \cup \Psi (N^{(j)}))$ which is a neighborhood of the seam of gluing. 
Here, define $\eta_\beta := (\eta_\beta^1,\ldots, \eta_\beta^n)$ as 
$$
\eta_\beta^1  (x):= \left\{ \begin{alignedat}{2}
    &-\sigma^{(j)} (x) ,&\quad & x \in  {(h^{(j)})}^{-1} (V_\beta)  \\ 
    & \sigma^{(j)} (\Psi^{-1} (x)), & \quad & x \in  \Psi \left( {(h^{(j)})}^{-1} (V_\beta) \right) ,
\end{alignedat} \right.
$$ 
and 
$$
\eta_\beta^i  (x):= \left\{ \begin{alignedat}{2}
    & \zeta_\beta^i (h^{(j)} (x)) ,&\quad & x \in  {(h^{(j)})}^{-1} (V_\beta) \\ 
    & \zeta_\beta^i (h^{(j)} (\Psi^{-1} (x))), & \quad & x \in   \Psi \left( {(h^{(j)})}^{-1} (V_\beta) \right),
\end{alignedat} \right. 
$$
for $i=2,\ldots,n$. 
Then, it can be checked that $\tilde{\eta}_\beta := \eta_\beta\circ \pi^{-1}$ is a homeomorphism from $\tilde{V}_\beta$ to a subset of $\mathbb{R}^n$. 
Therefore, the proof of (ii) is completed by showing that
\begin{equation}
    \label{eq:atlas}
    \left\{ (\tilde{U}_\alpha, \tilde{\vartheta}_\alpha), (\tilde{V}_\beta, \tilde{\eta}_\beta) \right\}_{\alpha\in A^{(j)},\beta\in B^{(j)}, j\in J}
\end{equation}
is an atlas endowing $\tilde{M}$ with a $C^r$ structure. 
As mentioned, the transition map $\tilde{\vartheta}_\alpha \circ \tilde{\vartheta}_{\alpha'}^{-1}$ is obviously $C^r$ for $\alpha \in A^{(j)}$. 
Similarly, for $\beta \in B^{(j)}$, the transition maps $\tilde{\vartheta}_\alpha \circ \tilde{\eta}_{\beta}^{-1}$ and $\tilde{\eta}_\beta \circ \tilde{\vartheta}_{\alpha}^{-1}$ are $C^r$ for both cases of $\alpha \in A^{(j)}$ and $\alpha \in A^{(j+1)}$, because $\varphi_{t}^{(j)}, \sigma^{(j)}, h^{(j)}, \Psi|_{N^{(j)}}$, and $\zeta_\beta$ are $C^r$ on $M$, and $\pi$ is a $C^r$ diffeomorphism. 
To derive the transition map $\tilde{\eta}_\beta \circ \tilde{\eta}_{\beta'}^{-1}$ for $\beta, \beta'\in B^{(j)}$, one needs ${\eta}_{\beta}^{-1}$, which can be obtained by
\begin{equation}
\nonumber
{\eta}_\beta^{-1}(y) = \left\{ \begin{alignedat}{2}
&\varphi_{y^1}^{(j)} (\zeta_\beta^{-1} (y^2,\ldots,y^n)),&\quad & y^1 < 0 \\
&\varphi_{y^1}^{(j+1)} (\Psi (\zeta_\beta^{-1} (y^2,\ldots,y^n))),&\quad & y^1 \geq 0, 
\end{alignedat} \right.
\end{equation}
where $y = (y^1, y^2,\ldots, y^n) \in \mathbb{R}^n$. 
Note that $y^1$ corresponds to the time of the flow. 
From the obvious relationship $\tilde{\eta}_\beta \circ \tilde{\eta}_{\beta'}^{-1} = \eta_\beta \circ \eta_{\beta'}^{-1}$, one can obtain
$$\begin{aligned}
&\tilde{\eta}_\beta^1 \circ \tilde{\eta}_{\beta'}^{-1} (y) \\
&= \left\{ \begin{alignedat}{2}
    &-\sigma^{(j)} \circ \varphi_{y^1}^{(j)} (\zeta_{\beta'}^{-1} (y^2,\ldots,y^n)) ,&\quad& y^1 < 0 \\ 
    &\sigma^{(j)} \circ \Psi^{-1} \circ \varphi_{y^1}^{(j+1)} \circ \Psi (\zeta_{\beta'}^{-1} (y^2,\ldots,y^n)), &\quad& y^1 \geq 0
\end{alignedat} \right. \\
&= \left\{ \begin{alignedat}{2}
    &-\sigma^{(j)} \circ \varphi_{y^1}^{(j)} (\zeta_{\beta'}^{-1} (y^2,\ldots,y^n)) ,&\quad& y^1 < 0 \\ 
    &\sigma^{(j)} \circ \Psi^{-1} \circ \Psi \circ \varphi_{-y^1}^{(j)} (\zeta_{\beta'}^{-1} (y^2,\ldots,y^n)) ,&\quad& y^1 \geq 0
\end{alignedat} \right. \\
&= y^1 \\ 
\end{aligned}
$$
$$
\begin{aligned}
&\tilde{\eta}_\beta^i \circ \tilde{\eta}_{\beta'}^{-1} (y) \\
&= \left\{ \begin{alignedat}{2}
   & \zeta_\beta^i \circ h^{(j)} \circ \varphi_{y^1}^{(j)} (\zeta_{\beta'}^{-1} (y^2,\ldots,y^n)) ,&~\, & y^1 < 0 \\ 
    &\zeta_\beta^i \circ h^{(j)}  \circ \Psi^{-1} \circ \varphi_{y^1}^{(j+1)} \circ \Psi (\zeta_{\beta'}^{-1} (y^2,\ldots,y^n)) ,& ~\, & y^1 \geq 0
\end{alignedat} \right. \\
&= \left\{ \begin{alignedat}{2}
   & \zeta_\beta^i \circ h^{(j)} \circ \varphi_{y^1}^{(j)} (\zeta_{\beta'}^{-1} (y^2,\ldots,y^n)) ,&\quad&  y^1 < 0 \\ 
   & \zeta_\beta^i \circ h^{(j)}  \circ \Psi^{-1} \circ \Psi \circ \varphi_{-y^1}^{(j)} (\zeta_{\beta'}^{-1} (y^2,\ldots,y^n)) ,&\quad&  y^1 \geq 0
\end{alignedat} \right. \\
&= \zeta^{i}_\beta \circ \zeta_{\beta'}^{-1} (y^2,\ldots,y^n). \\ 
\end{aligned}
$$
Here, the transformations from the 1st to the 2nd row and from the 4th to the 5th row follow from Eq. \eqref{eq:lem3}, the 2nd to the 3rd row from Eq. \eqref{eq:lem1}, and the 5th to the 6th from Eq. \eqref{eq:lem2}. 
Therefore, $\tilde{\eta}_\beta \circ \tilde{\eta}_{\beta'}^{-1}$ is $C^r$, proving that the atlas \eqref{eq:atlas} endows $\tilde{M}$ with a $C^r$ structure. 

(iii) For all $t>0$, $\tilde{\varphi}_t$ satisfying Eq. \eqref{eq:conjugacy} is uniquely determined \cite[Theorem A.31]{Lee_smooth}. 
Differentiating Eq. \eqref{eq:conjugacy} with respect to $t$ at $t=0$ yields $
\left. \frac{{\rm d}}{{\rm d}t} \tilde{\varphi}_t ( \pi (x)) \right|_{t=0} = ({\rm D}\pi F)_x
$, implying that the infinitesimal generator of $\tilde{\varphi}_t$ corresponds to ${\rm D}\pi F$. 
This can be verified to correspond to the pushforward of $F$ by $\pi$. 
Thus, the proof of (iii) is completed by showing the pushforward $\tilde{F}:= {\rm D}\pi F$ is $C^r$, which is equivalent to showing that the component functions of $\tilde{F}$, based on the atlas \eqref{eq:atlas}, are $C^r$. 
The component functions in terms of $\tilde{\vartheta}_\alpha$ are obviously $C^r$. 
The differentiabilities of the component functions in terms of $\tilde{\eta}_\beta$ can be confirmed through Proposition \ref{prop:smooth_vectorfield}.
First, we can obtain the following relationship
$$
\begin{aligned}
    &{\rm d} \tilde{\eta}_\beta^i|_{\tilde{\eta}_\beta^{-1} (y)} (\tilde{F}) 
    = {\rm d}  (\eta_{\beta}^i \circ \pi^{-1})_{\tilde{\eta}_\beta^{-1}(y)} ({\rm D} \pi F) 
    \\&= {\rm d} {\eta_{\beta}^i}|_{ \pi^{-1} \circ\tilde{\eta}_\beta^{-1}(y)} ({\rm D} \pi^{-1}{\rm D}\pi F) 
    = {\rm d} {\eta_\beta^i}_{,\eta_\beta^{-1}(y)} (F), 
\end{aligned}
$$
through Eq. \eqref{eq:pullback}. 
Then, from \eqref{eq:time-to-impact-dif}-\eqref{eq:Psi_inv_dif}, we can calculate for the case $i=1$ as
$$
\begin{aligned}
     {\rm d} \tilde{\eta}_\beta^1|_{\tilde{\eta}_\beta^{-1} (y)}
    &= \left\{ \begin{alignedat}{2}
        & -{\rm d} \sigma^{(j)}|_{\eta^{-1}_\beta(y)} (F) ,&\quad& y^1 < 0\\ 
        & {\rm d} (\sigma^{(j)}\circ \Psi^{-1})_{\eta_\beta^{-1}(y)} (F),&\quad& y^1 \geq 0
    \end{alignedat} \right. \\
    &= \left\{ \begin{alignedat}{2}
        & -{\rm d} \sigma^{(j)}|_{\eta^{-1}_\beta(y)} (F) ,&\quad &y^1 < 0\\ 
        & {\rm d} \sigma^{(j)}|_{\Psi^{-1}\circ \eta_\beta^{-1}(y)} ({\rm D}\Psi^{-1} F),&\quad& y^1 \geq 0
    \end{alignedat} \right. \\
    &= \left\{ \begin{alignedat}{2}
        & -{\rm d} \sigma^{(j)}|_{\eta^{-1}_\beta(y)} (F) ,&\quad& y^1 < 0\\ 
        & {\rm d} \sigma^{(j)}|_{\Psi^{-1}\circ \eta_\beta^{-1}(y)} (- F),&\quad& y^1 \geq 0
    \end{alignedat} \right. \\ 
    &= 1.
\end{aligned}
$$
Here, the transformation from the 1st to the 2nd row uses Eq. \eqref{eq:pullback}, the 2nd to the 3rd row uses Eq. \eqref{eq:Psi_inv_dif}, and the 3rd to 4th row uses Eq. \eqref{eq:time-to-impact-dif}. 
Similarly, for the case $i=2,\ldots,n$, we can calculate as
$$
\begin{aligned}
     {\rm d} \tilde{\eta}_\beta^i|_{\tilde{\eta}_\beta^{-1} (y)}
    &= \left\{ \begin{alignedat}{2}
        & {\rm d} (\zeta_\beta^i \circ h^{(j)})_{\eta^{-1}_\beta(y)} (F) ,&\quad& y^1 < 0\\ 
        & {\rm d} (\zeta_\beta^i \circ h^{(j)}\circ \Psi^{-1})_{\eta_\beta^{-1}(y)} (F),&\quad& y^1 \geq 0
    \end{alignedat} \right. \\
    &= \left\{ \begin{alignedat}{2}
        & {\rm d} (\zeta_\beta^i )_{h\circ \eta^{-1}_\beta(y)} ({\rm D}h^{(j)} F) ,&\quad& y^1 < 0\\ 
        & {\rm d} (\zeta_\beta^i )_{h\circ \Psi^{-1} \circ \eta_\beta^{-1}(y)} ({\rm D}h^{(j)} {\rm D} \Psi^{-1} F),&\quad& y^1 \geq 0
    \end{alignedat} \right. \\
    &= 0.
\end{aligned}
$$
Here, the transformation from the 2nd to the 3rd row follows from Eqs. \eqref{eq:hF} and \eqref{eq:Psi_inv_dif}. 
Therefore, we conclude that the component functions in terms of $\tilde{\eta}_\beta$ are $C^r$, which completes the proof of (iii).
\end{proof}

\subsection{Proof of Lemma 2}
Before proving Lemma \ref{lem:existence_of_vector}, we establish the following lemma. 
\begin{lemma}
\label{lem:diffe}
If $G^{(j)}$ is diffeomorphic to a subset of $\mathbb{R}^{n-1}$ for all $j\in J$, then there exist $n-1$ vector fields $F_2,\ldots,F_n$ satisfying (a) and (b) of Lemma \ref{lem:existence_of_vector}. 
\end{lemma}
\begin{proof}[Proof of Lemma \ref{lem:diffe}]
    Let $\mu:G^{(j)}\to U \subset \mathbb{R}^{n-1}$ be the diffeomorphism, and define the vector fields $F_i = {\rm D}\mu^{-1} (\partial / \partial y_i)$ as pushforwards, where $(\partial / \partial y_2), \ldots, (\partial / \partial y_n)$ are orthogonal basis vectors of ${\rm T}U$. 
    Then, $F_2, \ldots, F_n$ satisfies (a) and (b) of Lemma \ref{lem:existence_of_vector} by \cite[Theorem 9.46]{Lee_smooth}.
\end{proof}
\begin{proof}[Proof of Lemma \ref{lem:existence_of_vector}]
    One can suppose $G^{(1)}$ is a Poincare section of the hybrid limit cycle. 
    Let $P_{G^{(1)}}:G^{(1)}\mapsto P(G^{(1)})$ be the associated Poincare map. 
    Then, from Assumtion \ref{ass:GESHLC}, the discrete dynamical system $(G^{(1)},P_{G^{(1)}})$ possesses a globally asymptotically stable fixed point $x^* \in {G^{(1)}}$. 
    Then, there exists a chart $(U,\vartheta)$ such that $U$ is a neighborhood of $x^*$ and forward invariant under the mapping $P|_{G^{(1)}}$, implying that $U$ is diffeomorphic to a subset of $\mathbb{R}^{n-1}$. 
    A set $U'$ with $P_{G^{(1)}}(U') \subseteq U$ is also diffeomorphic to a subset of $\mathbb{R}^{n-1}$ since $P_{G^{(1)}}$ is the diffeomorphism. 
    One can choose $U'$ as satisfying $U' \supset U$, since $P_{G^{(1)}}$ has a globally asymptotically stable fixed point. 
    This extension process can be repeated until $U' = G^{(1)}$, which implies that $G^{(1)}$ is diffeomorphic to a subset of $\mathbb{R}^{n-1}$
    Similarly, $G^{(2)},\ldots,G^{(|J|)}$ are diffeomorphic to subsets of $\mathbb{R}^{n-1}$. 
    Thus, Lemma \ref{lem:diffe} completes the proof. 
\end{proof}

\subsection{Proof of Theorem 1}
Before proceeding the proof, we establish the following fact: 
given that the vector fields $F,F_2,\ldots,F_n$ commute, and that for any $z \in G$, ${\rm T}_z M = {\rm span}\{ F_z,{F_2}_{,z},\ldots, {F_n}_{,z} \}$ and ${\rm T}_z G = {\rm span}\{{F_2}_{,z},\ldots, {F_n}_{,z} \}$, there exists a local coordinate expressed in terms of $(y^1,\ldots,y^n)$ such that $F=(\partial/ \partial y^1),F_2 = (\partial/ \partial y^2),\ldots, F_n = (\partial / \partial y^n)$ as shown in \cite[Theorem 9.46]{Lee_smooth}.
Since the proof of $k=0$ is straightforward from Remark \ref{rem:projection}, we begin by considering $k = 1$ and then proceed to the case $k \geq 2$.

\begin{proof}[Proof of $(\Longrightarrow)$ of Eq. \eqref{eq:iff}]
Choose an arbitrary $f\in \tilde{C}^k (M)$. 
Based on the $C^r$ structure endowed by the atlas \eqref{eq:atlas}, it suffices to show that $f\circ \pi^{-1} \circ \tilde{\vartheta}_\alpha^{-1}$ and $f\circ \pi^{-1} \circ \tilde{\eta}_\beta^{-1}$ are $C^k$ on $\tilde{U}_\alpha$ and $\tilde{V}_\beta$, respectively. 
The former is straightforward from the condition $f|_{M^{(j)}} \in C^k (M^{(j)})$ of Definition \ref{def:tildeC}. 
For the letter, we have
\begin{equation}
\label{eq:f_pi_inv_tilde_eta}
\begin{aligned}[b]
&f\circ \pi^{-1} \circ \tilde{\eta}_\beta^{-1} (y) \\
&= \left\{ \begin{alignedat}{2}
& f \circ \varphi_{y^1}^{(j)} (\zeta_\beta^{-1} (y^2,\ldots,y^n)),&\quad & y^1 < 0 \\
&f\circ  \varphi_{y^1}^{(j+1)} \circ \Psi (\zeta_\beta^{-1} (y^2,\ldots,y^n)),&\quad & y^1 \geq 0 
\end{alignedat} \right. \\
&= \left\{ \begin{alignedat}{2}
& f \circ \varphi_{y^1}^{(j)} (\zeta_\beta^{-1} (y^2,\ldots,y^n)),&\quad & y^1 < 0 \\
&f\circ \Psi \circ \varphi_{-y^1}^{(j)} (\zeta_\beta^{-1} (y^2,\ldots,y^n)),&\quad & y^1 \geq 0.
\end{alignedat} \right.
\end{aligned}
\end{equation}
The last row comes from Eq. \eqref{eq:lem3}. 
Thus, it is obvious that $f\circ \pi^{-1} \circ \tilde{\eta}_\beta^{-1}$ is $C^k$ on $\tilde{V}_\beta \backslash \{ y \in \mathbb{R}^n ~|~y^1=0\}$. 

To show the differentiability at $y^1 = 0$, we examine the continuity of the derivatives as $y^1 \uparrow 0$ and $y^1 \downarrow 0$. 
The continuity of the 0th derivative, (i.e., $f\circ \pi^{-1} \circ \tilde{\eta}_\beta^{-1}$ itself), follows directly from Eq. \eqref{eq:f_pi_inv_tilde_eta}. 
Using the relation $\frac{{\rm d}}{{\rm d}t}f\circ \varphi_t^{(j)} (x) = {\rm d}f_{\varphi_t^{(j)} (x)} (F)$ \cite[Proposition 11.23]{Lee_smooth}, differentiating Eq. \eqref{eq:f_pi_inv_tilde_eta} with respect to $y^1$ yields 
\begin{equation}
\label{eq:derivative_1}
\begin{aligned}[b]
\frac{\partial}{\partial y^1} & f\circ \pi^{-1} \circ \tilde{\eta}_\beta^{-1} (y) \\
&= \left\{ \begin{alignedat}{2}
& ~\quad {\rm d}f_{\varphi_{y^1}^{(j)} \circ \zeta_\beta^{-1} (y^2,\ldots, y^n)} (F) ,&\quad & y^1 < 0 \\
&- {\rm d} (f\circ \Psi)_{\varphi_{-y^1}^{(j)} \circ \zeta_\beta^{-1} (y^2,\ldots, y^n)} (F) ,&\quad & y^1 \geq 0 
\end{alignedat} \right. \\
&= \left\{ \begin{alignedat}{2}
& ~\quad{\rm d}f_{\varphi_{y^1}^{(j)} \circ \zeta_\beta^{-1} (y^2,\ldots, y^n)} (F) ,&\quad & y^1 < 0 \\
&- {\rm d} f_{\Psi \circ\varphi_{-y^1}^{(j)} \circ \zeta_\beta^{-1} (y^2,\ldots, y^n)} ({\rm D}\Psi F) ,&\quad & y^1 \geq 0 
\end{alignedat} \right. \\
&= \left\{ \begin{alignedat}{2}
& F f_{\varphi_{y^1}^{(j)} \circ\zeta_\beta^{-1} (y^2,\ldots, y^n)} ,&\quad & y^1 < 0 \\
&F f_{\Psi \circ\varphi_{-y^1}^{(j)} \circ\zeta_\beta^{-1} (y^2,\ldots, y^n)} ,&\quad & y^1 \geq 0. 
\end{alignedat} \right. 
\end{aligned}
\end{equation}
Similarly, using $F_i = (\partial/\partial y^i)$, the differentiation with respect to $y^i$ yields 
\begin{equation}
\label{eq:derivative_i}
\begin{aligned}[b]
\frac{\partial}{\partial y^i} & f\circ \pi^{-1} \circ \tilde{\eta}_\beta^{-1} (y) \\
&= \left\{ \begin{alignedat}{2}
&  {\rm d}(f\circ \varphi_{y^1}^{(j)})_{ \zeta_\beta^{-1} (y^2,\ldots, y^n)} (F_i) ,&\quad & y^1 < 0 \\
& {\rm d} (f\circ \Psi \circ \varphi_{-y^1}^{(j)})_{ \zeta_\beta^{-1} (y^2,\ldots, y^n) } (F_i) ,&\quad & y^1 \geq 0 
\end{alignedat} \right. \\
&= \left\{ \begin{alignedat}{2}
&  {\rm d}f_{\varphi_{y^1}^{(j)} \circ \zeta_\beta^{-1} (y^2,\ldots, y^n)} ({\rm D}\varphi_{y^1}^{(j)} F_i) ,&\quad & y^1 < 0 \\
&  {\rm d} f_{ \Psi \circ \varphi_{-y^1}^{(j)} \circ \zeta_\beta^{-1} (y^2,\ldots, y^n) } ({\rm D}\Psi {\rm D}\varphi_{-y^1}^{(j)} F_i) ,&\quad & y^1 \geq 0 
\end{alignedat} \right. \\
&= \left\{ \begin{alignedat}{2}
& {\rm D}\varphi_{y^1}^{(j)} F_i f_{\varphi_{y^1}^{(j)} \circ \zeta_\beta^{-1} (y^2,\ldots, y^n)}  ,&\quad & y^1 < 0 \\
&  {\rm D}\Psi {\rm D}\varphi_{-y^1}^{(j)} F_i  f_{ \Psi \circ \varphi_{-y^1}^{(j)} \circ \zeta_\beta^{-1} (y^2,\ldots, y^n) }  ,&\quad & y^1 \geq 0 . 
\end{alignedat} \right. 
\end{aligned}
\end{equation}
For $y^1=0$, since $\varphi_{y^1}^{(j)} ={\rm i.d.}, ~{\rm D}\varphi_{y^1}^{(j)} ={\rm i.d.}$, and $\Psi\circ \varphi_{-y^1}^{(j)}\circ \zeta_\beta^{-1} =  \Psi\circ \zeta_\beta^{-1} = R\circ \zeta_\beta^{-1}$, taking the limits $y^1 \uparrow 0$ and $y^1 \downarrow 0$ of Eqs. \eqref{eq:derivative_1} and \eqref{eq:derivative_i} yield
\begin{equation}
\nonumber
\begin{aligned}[b]
    \frac{\partial}{\partial y^1} & f\circ \pi^{-1} \circ \tilde{\eta}_\beta^{-1} (y) \\
    &= \left\{ \begin{alignedat}{2}
        & L_F f (\zeta_\beta^{-1} (y^2,\ldots,y^n)) ,&\quad & as ~ y^1 \uparrow 0 \\ 
        & L_F f (R (\zeta_\beta^{-1} (y^2,\ldots,y^n))) ,&\quad & as ~ y^1 \downarrow 0 ,
    \end{alignedat}\right. 
\end{aligned}
\end{equation}
\begin{equation}
\nonumber
\begin{aligned}[b]
    \frac{\partial}{\partial y^i} & f\circ \pi^{-1} \circ \tilde{\eta}_\beta^{-1} (y) \\
    &= \left\{ \begin{alignedat}{2}
        & L_{ F_i} f  (\zeta_\beta^{-1} (y^2,\ldots,y^n)) ,&\quad & as~ y^1 \uparrow 0 \\ 
        & L_{{\rm D}\Psi  F_i} f (R  (\zeta_\beta^{-1} (y^2,\ldots,y^n))) ,&\quad & as~ y^1 \downarrow 0 .
    \end{alignedat}\right. 
\end{aligned}    
\end{equation}
Here, we express the vector fields, which act as differential operators, in terms of Lie derivatives. 
From the conditions \eqref{eq:lie}, we can verify that the 1st derivatives of $f\circ \pi^{-1} \circ \tilde{\eta}_\beta^{-1}$ are continuous at $y^1=0$. 

Let $k \geq 2$.
To calculate the 2nd derivatives, we introduce functions $f_{y^i}~(i=1,\ldots,n)$ defined on a subset of $N^{(j)}\cup \Psi(N^{(j)})$ such that $f_{y^1} := Ff$ and 
$$ f_{y^i} (x):= \left\{
\begin{alignedat}{2}
    &{\rm D} \varphi_{-\sigma^{(j)}(x)}^{(j)} F_i f|_x ,&\quad& x\in N^{(j)} \\
    &{\rm D} \Psi {\rm D} \varphi_{\sigma^{(j)}(\Psi^{-1} (x))}^{(j)} F_i f|_x , &\quad &x\in \Psi(N^{(j)}).
\end{alignedat}\right.
$$
One can verify that $\frac{\partial}{\partial y^i} f\circ \pi^{-1} \circ \tilde{\eta}_\beta^{-1} = f_{y^i}\circ \pi^{-1} \circ \tilde{\eta}_\beta^{-1}$. 
Using these functions, we can calculate the 2nd derivatives of $f\circ \pi^{-1} \circ \tilde{\eta}_\beta^{-1}$ at $y^1 = 0$ and obtain results consistent with those of the first derivatives.
By applying this process iteratively, we can show the continuities of the high order derivatives of $f\circ \pi^{-1} \circ \tilde{\eta}_\beta^{-1}$ at $y^1=0$ for all orders up to $k$. 
Therefore, $f\circ \pi^{-1} $ is $C^k$ across all charts, resulting in $f \circ \pi^{-1} \in {C}^k(\tilde{M})$. 
\end{proof}

\begin{proof}[Proof of ($\Longleftarrow$) of Eq. \eqref{eq:iff}]
 Assuming $f\notin {\tilde{C}}^k (M)$, similar reasoning as in ($\Longrightarrow$) leads to $f \circ \pi^{-1} \notin {C}^k(\tilde{M})$, implying the contraposition holds.
\end{proof}

\begin{proof}[Proof of Eq. \eqref{eq:invariance}]
    $f \in {\tilde{C}}^k(M) \iff f \circ \pi^{-1} \in {C}^k(M) \Longrightarrow f \circ \pi^{-1} \circ \tilde{\varphi}_t \in {C}^k(M) \iff f \circ \pi^{-1} \circ \tilde{\varphi}_t \circ \pi \in {\tilde{C}}^k(M) \iff f \circ \varphi_t \in {\tilde{C}}^k(M) \iff U_t f \in {\tilde{C}}^k(M) .$
\end{proof}

\subsection{Proof of Theorem 2}
Before proving Theorem \ref{thm:existence_KEF}, we establish the following lemma.
\begin{lemma}
    \label{lem:existence_of_LC}
    Consider $H$ to be a hybrid dynamical system satisfying Assumptions \ref{ass:transverseG}-\ref{ass:GESHLC}. 
    Then, a $C^r$ dynamical system $(\tilde{M},\tilde{\varphi}_t)$ linked by Eq. \eqref{eq:conjugacy} has a globally asymptotically stable limit cycle $\pi(\Gamma)$, which retains the same period and the Floquet exponents as the hybrid case. 
\end{lemma}
\begin{proof}
    It is straightforward that $\pi (\Gamma)$ is a periodic orbit of $(\tilde{M},\tilde{\varphi}_t)$ with the same period as $\Gamma$. 
    Consider any $\tilde{x}\in \tilde{M}$, then we can take $x=\pi^{-1} (\tilde{x}) \in M$. 
    By Assumption \ref{ass:GESHLC} and Definition \ref{def:GEHLC}, there exists a Poincare map $P_\Sigma$ with a Poincare section $\Sigma \in M^{(j)}$ such that $x$ belongs $\Sigma$ and the iterates of $P_\Sigma$ to $x$ converges to $x^* \in \Gamma \cap \Sigma$. 
    For the system $(\tilde{M},\tilde{\varphi}_t)$, one can consider a Poincare map $\tilde{P}_{\tilde{\Sigma}}$ with a Poincare section $\tilde{\Sigma} = \pi (\Sigma)$ such that 
    \begin{equation}
        \label{eq:lem5}
        \tilde{P}_{\tilde{\Sigma}} \circ \pi = \pi \circ P_\Sigma.
    \end{equation} 
    The iterates of $\tilde{P}_{\tilde{\Sigma}}$ to $\tilde{x}$ converges to $\pi(x^*) \in \pi(\Gamma) \cap \pi(\Sigma)$, implying $\pi(\Gamma)$ is a globally asymptotically stable limit cycle. 
    
    For chosen $\Sigma$, $\pi|_{\Sigma}$ is diffeomorphic to its image from Lemma \ref{lem:smoothing}. 
    Then, ${\rm D} \tilde{P}_{\tilde{\Sigma}}$ and ${\rm D} P_\Sigma$ share the same eigenvalues since they are linked by Eq. \eqref{eq:lem5}. 
    By recalling the correspondence between eigenvalues of the derivative of the Poincare map and the Floquet exponents, the proof is completed.
\end{proof}
\begin{proof}[Proof of Theorem 2]
    According to Lemma \ref{lem:existence_of_LC}, the $C^r$ dynamical system $(\tilde{M},\tilde{\varphi}_t)$ has a globally asymptotically stable limit cycle with period $\tau$, fundamental frequency $\omega = 2\pi/\tau$, and Floquet exponents $\nu_2, \ldots, \nu_n$. 
    Here, we introduce the Koopman operator $\tilde{U}_t$ for the dynamical system $(\tilde{M},\tilde{\varphi}_t)$ as $\tilde{U}_t \tilde{f} = \tilde{f} \circ \tilde{\varphi}_t,~f\in C^r (\tilde{M})$. 
    Since $\nu_2, \ldots, \nu_n$ satisfy $r$-nonresonant and spectral spread conditions, there exist eigenfunctions $\tilde{\phi}_{{\rm i}\omega},~\phi_{\nu_2},\ldots, \phi_{\nu_n}$ of $\tilde{U}_t$ such that 
    $$ \tilde{U}_t \tilde{\phi}_{{\rm i}\omega} =  \tilde{\phi}_{{\rm i}\omega}\circ \tilde{\varphi}_t = {\rm e}^{{\rm i}\omega t} \tilde{\phi}_{{\rm i}\omega},\quad \tilde{\phi}_{{\rm i}\omega}\in {C}^r(\tilde{M})\backslash \{ 0 \}
    $$
    $$ \tilde{U}_t \tilde{\phi}_{\nu_i} = \tilde{\phi}_{\nu_i}\circ \tilde{\varphi}_t = {\rm e}^{\nu_i t} \tilde{\phi}_{\nu_i},\quad \tilde{\phi}_{\nu_i}\in {C}^r(\tilde{M})\backslash \{ 0 \}
    $$
    and are uniquely determined \cite[Proposition 3 and 7]{kvalheim2021existence}. 
    According to theorem \ref{thm:smooth}, $\phi_{{\rm i}\omega}:= \tilde{\phi}_{{\rm i}\omega} \circ \pi^{-1}\in \tilde{C}^r(M)$ and $\phi_{\nu_i} :=\tilde{\phi}_{\nu_i}\circ \pi^{-1} \in \tilde{C}^r(M)$.
    One can calculate that for $\lambda = {\rm i}\omega, \nu_2,\ldots,\nu_n$,  $$U_t \phi_\lambda = \tilde{\phi}_\lambda \circ \pi^{-1} \circ \varphi_t = \tilde{\phi}_\lambda \circ \tilde{\varphi}_t \circ \pi^{-1} = {\rm e}^{\lambda t} \tilde{\varphi}_\lambda \circ \pi^{-1} = {\rm e}^{\lambda t} {\phi}_\lambda.$$ 
    Thus, $\phi_{{\rm i}\omega},~\phi_{\nu_2},\ldots,\phi_{\nu_n}$ are Koopman eigenfunctions such that \eqref{eq:KEFomega} and \eqref{eq:KEFnu} hold.     
\end{proof}
}

\bibliographystyle{IEEEtran}
\bibliography{Refs}

\begin{thebibliography}{10}
\providecommand{\url}[1]{#1}
\csname url@samestyle\endcsname
\providecommand{\newblock}{\relax}
\providecommand{\bibinfo}[2]{#2}
\providecommand{\BIBentrySTDinterwordspacing}{\spaceskip=0pt\relax}
\providecommand{\BIBentryALTinterwordstretchfactor}{4}
\providecommand{\BIBentryALTinterwordspacing}{\spaceskip=\fontdimen2\font plus
\BIBentryALTinterwordstretchfactor\fontdimen3\font minus \fontdimen4\font\relax}
\providecommand{\BIBforeignlanguage}[2]{{%
\expandafter\ifx\csname l@#1\endcsname\relax
\typeout{** WARNING: IEEEtran.bst: No hyphenation pattern has been}%
\typeout{** loaded for the language `#1'. Using the pattern for}%
\typeout{** the default language instead.}%
\else
\language=\csname l@#1\endcsname
\fi
#2}}
\providecommand{\BIBdecl}{\relax}
\BIBdecl

\bibitem{van2007introduction}
A.~J. Van Der~Schaft and H.~Schumacher, \emph{An introduction to hybrid dynamical systems}.\hskip 1em plus 0.5em minus 0.4em\relax springer, 2007, vol. 251.

\bibitem{goebel2009hybrid}
R.~Goebel, R.~G. Sanfelice, and A.~R. Teel, ``Hybrid dynamical systems,'' \emph{IEEE control systems magazine}, vol.~29, no.~2, pp. 28--93, 2009.

\bibitem{simic2001structural}
S.~N. Simic, K.~H. Johansson, J.~Lygeros, and S.~Sastry, ``Structural stability of hybrid systems,'' in \emph{2001 European Control Conference (ECC)}.\hskip 1em plus 0.5em minus 0.4em\relax IEEE, 2001, pp. 3858--3863.

\bibitem{simic2002hybrid}
S.~N. Simic, S.~Sastry, K.~H. Johansson, and J.~Lygeros, ``Hybrid limit cycles and hybrid poincar{\'e}-bendixson,'' \emph{IFAC Proceedings Volumes}, vol.~35, no.~1, pp. 197--202, 2002.

\bibitem{mauroy2020koopman}
A.~Mauroy, I.~Mezi{\'c}, and Y.~Susuki, Eds., \emph{{The Koopman Operator in Systems and Control: Concepts, Methodologies, and Applications}}.\hskip 1em plus 0.5em minus 0.4em\relax Springer, 2020.

\bibitem{mauroy2013isostable}
A.~Mauroy, I.~Mezi{\'c}, and J.~Moehlis, ``Isostables, isochrons, and {K}oopman spectrum for the action--angle representation of stable fixed point dynamics,'' \emph{Physica D: Nonlinear Phenomena}, vol. 261, pp. 19--30, 2013.

\bibitem{mauroy2012use}
A.~Mauroy and I.~Mezi{\'c}, ``On the use of {F}ourier averages to compute the global isochrons of (quasi) periodic dynamics,'' \emph{{CHAOS}: {An} {I}nterdisciplinary {J}ournal of {N}onlinear {S}cience}, vol.~22, no.~3, p. 033112, 2012.

\bibitem{shirasaka2017phase_isostable}
S.~Shirasaka, W.~Kurebayashi, and H.~Nakao, ``{P}hase-amplitude reduction of transient dynamics far from attractors for limit-cycling systems,'' \emph{Chaos: An Interdisciplinary Journal of Nonlinear Science}, vol.~27, no.~2, p. 023119, 2017.

\bibitem{monga2019phase}
B.~Monga, D.~Wilson, T.~Matchen, and J.~Moehlis, ``Phase reduction and phase-based optimal control for biological systems: a tutorial,'' \emph{Biological cybernetics}, vol. 113, no.~1, pp. 11--46, 2019.

\bibitem{wilson2016isostable}
D.~Wilson and J.~Moehlis, ``Isostable reduction of periodic orbits,'' \emph{Physical Review E}, vol.~94, no.~5, p. 052213, 2016.

\bibitem{kvalheim2021existence}
M.~D. Kvalheim and S.~Revzen, ``Existence and uniqueness of global {K}oopman eigenfunctions for stable fixed points and periodic orbits,'' \emph{Physica D: Nonlinear Phenomena}, vol. 425, p. 132959, 2021.

\bibitem{mezic2020spectrum}
I.~Mezi{\'c}, ``Spectrum of the {K}oopman operator, spectral expansions in functional spaces, and state-space geometry,'' \emph{Journal of Nonlinear Science}, vol.~30, no.~5, pp. 2091--2145, 2020.

\bibitem{govindarajan2016operator}
N.~Govindarajan, H.~Arbabi, L.~Van~Blargian, T.~Matchen, E.~Tegling \emph{et~al.}, ``An operator-theoretic viewpoint to non-smooth dynamical systems: Koopman analysis of a hybrid pendulum,'' in \emph{2016 IEEE 55th Conference on Decision and Control (CDC)}.\hskip 1em plus 0.5em minus 0.4em\relax IEEE, 2016, pp. 6477--6484.

\bibitem{gerlach2020koopman}
A.~R. Gerlach, A.~Leonard, J.~Rogers, and C.~Rackauckas, ``The {K}oopman expectation: An operator theoretic method for efficient analysis and optimization of uncertain hybrid dynamical systems,'' \emph{arXiv preprint arXiv:2008.08737}, 2020.

\bibitem{shirasaka2017phase_hybrid}
S.~Shirasaka, W.~Kurebayashi, and H.~Nakao, ``Phase reduction theory for hybrid nonlinear oscillators,'' \emph{Physical Review E}, vol.~95, no.~1, p. 012212, 2017.

\bibitem{grizzle2001asymptotically}
J.~W. Grizzle, G.~Abba, and F.~Plestan, ``Asymptotically stable walking for biped robots: Analysis via systems with impulse effects,'' \emph{IEEE Transactions on automatic control}, vol.~46, no.~1, pp. 51--64, 2001.

\bibitem{morris2005restricted}
B.~Morris and J.~W. Grizzle, ``A restricted poincar{\'e} map for determining exponentially stable periodic orbits in systems with impulse effects: Application to bipedal robots,'' in \emph{Proceedings of the 44th IEEE Conference on Decision and Control}.\hskip 1em plus 0.5em minus 0.4em\relax IEEE, 2005, pp. 4199--4206.

\bibitem{hiskens2007switching}
I.~A. Hiskens and P.~B. Reddy, ``Switching-induced stable limit cycles,'' \emph{Nonlinear Dynamics}, vol.~50, pp. 575--585, 2007.

\bibitem{lou2015results}
X.~Lou, Y.~Li, and R.~G. Sanfelice, ``Results on stability and robustness of hybrid limit cycles for a class of hybrid systems,'' in \emph{2015 54th IEEE Conference on Decision and Control (CDC)}.\hskip 1em plus 0.5em minus 0.4em\relax IEEE, 2015, pp. 2235--2240.

\bibitem{simic2005towards}
S.~N. Simic, K.~H. Johansson, J.~Lygeros, and S.~Sastry, ``Towards a geometric theory of hybrid systems,'' \emph{Dynamics of Continuous, Discrete and Impulsive Systems Series B: Applications and Algorithms}, vol.~12, no. 5-6, pp. 649--687, 2005.

\bibitem{burden2015model}
S.~A. Burden, S.~Revzen, and S.~S. Sastry, ``Model reduction near periodic orbits of hybrid dynamical systems,'' \emph{IEEE Transactions on Automatic Control}, vol.~60, no.~10, pp. 2626--2639, 2015.

\bibitem{Lee_smooth}
J.~M. Lee, \emph{{Introduction to Smooth Manifolds}}.\hskip 1em plus 0.5em minus 0.4em\relax Springer, 2012.

\bibitem{liu2023non}
Z.~Liu, N.~Ozay, and E.~D. Sontag, ``On the non-existence of immersions for systems with multiple omega-limit sets,'' \emph{IFAC-PapersOnLine}, vol.~56, no.~2, pp. 60--64, 2023.

\bibitem{kvalheim2023linearizability}
M.~D. Kvalheim and P.~Arathoon, ``Linearizability of flows by embeddings,'' \emph{arXiv preprint arXiv:2305.18288}, 2023.

\bibitem{kutz2016dynamic}
J.~N. Kutz, S.~L. Brunton, B.~W. Brunton, and J.~L. Proctor, \emph{{Dynamic Mode Decomposition: Data-Driven Modeling of Complex Systems}}.\hskip 1em plus 0.5em minus 0.4em\relax SIAM, 2016.

\bibitem{katayama2024koopman}
N.~Katayama and Y.~Susuki, ``Koopman analysis of the singularly-perturbed van der pol oscillator,'' \emph{Chaos: An Interdisciplinary Journal of Nonlinear Science}, vol.~34, no.~9, 2024.

\end{thebibliography}

\end{document}